\documentclass[12pt,reqno]{amsart} 
\usepackage{amssymb,amscd,url}
\usepackage{color}    
\usepackage[all]{xy}  
\usepackage{bbm}      
\usepackage{extarrows}
\usepackage{constants} 

\begin{document}




\title[The Distribution Relation and Inverse Function Theorem]      
      {The Distribution Relation and Inverse Function Theorem in Arithmetic Geometry}
\date{\today}
\author[Y. Matsuzawa]{Yohsuke Matsuzawa}
\email{matsuzawa@math.brown.edu}
\address{Department of Mathematics, Box 1917
  Brown University, Providence, RI 02912 USA}
\author[J.H. Silverman]{Joseph H. Silverman}
\email{jhs@math.brown.edu}
\address{Department of Mathematics, Box 1917
  Brown University, Providence, RI 02912 USA.
  ORCID: 0000-0003-3887-3248}

\thanks{The first author's research is supported by a JSPS Overseas
  Research Fellowship.  The second author's research is supported by
  Simons Collaboration Grant \#712332}


\newcommand{\YOHSUKE}[1]{{\color{green} $\bigstar$ \textsf{{\bf Yohsuke:} [#1]}}}
\newcommand{\JOE}[1]{{\color{blue} $\bigstar$ \textsf{{\bf Joe:} [#1]}}}


\hyphenation{ca-non-i-cal semi-abel-ian}


\newtheorem{theorem}{Theorem}[section]
\newtheorem{lemma}[theorem]{Lemma}
\newtheorem{sublemma}[theorem]{Sublemma}
\newtheorem{conjecture}[theorem]{Conjecture}
\newtheorem{proposition}[theorem]{Proposition}
\newtheorem{corollary}[theorem]{Corollary}
\newtheorem{claim}{Claim}
\newtheorem*{theoremtemplate}{Theorem Template}

\theoremstyle{definition}
\newtheorem{definition}[theorem]{Definition}
\newtheorem*{intuition}{Intuition}
\newtheorem{example}[theorem]{Example}
\newtheorem{remark}[theorem]{Remark}
\newtheorem{question}[theorem]{Question}

\theoremstyle{remark}
\newtheorem*{acknowledgement}{Acknowledgements}


\newenvironment{notation}[0]{%
  \begin{list}%
    {}%
    {\setlength{\itemindent}{0pt}
     \setlength{\labelwidth}{4\parindent}
     \setlength{\labelsep}{\parindent}
     \setlength{\leftmargin}{5\parindent}
     \setlength{\itemsep}{0pt}
     }%
   }%
  {\end{list}}

\newenvironment{parts}[0]{%
  \begin{list}{}%
    {\setlength{\itemindent}{0pt}
     \setlength{\labelwidth}{1.5\parindent}
     \setlength{\labelsep}{.5\parindent}
     \setlength{\leftmargin}{2\parindent}
     \setlength{\itemsep}{0pt}
     }%
   }%
  {\end{list}}
\newcommand{\Part}[1]{\item[\upshape#1]}

\def\Case#1#2{%
\paragraph{\textbf{\boldmath Case #1: #2.}}\hfil\break\ignorespaces}

\renewcommand{\a}{\alpha}
\newcommand{\bfalpha}{{\boldsymbol{\alpha}}}
\renewcommand{\b}{\beta}
\newcommand{\bfbeta}{{\boldsymbol{\beta}}}
\newcommand{\g}{\gamma}
\renewcommand{\d}{\delta}
\newcommand{\e}{\epsilon}
\newcommand{\f}{\varphi}
\newcommand{\fhat}{\hat\varphi}
\newcommand{\bfphi}{{\boldsymbol{\f}}}
\renewcommand{\l}{\lambda}
\renewcommand{\k}{\kappa}
\newcommand{\lhat}{\hat\lambda}
\newcommand{\m}{\mu}
\newcommand{\bfmu}{{\boldsymbol{\mu}}}
\renewcommand{\o}{\omega}
\newcommand{\bfpi}{{\boldsymbol{\pi}}}
\renewcommand{\r}{\rho}
\newcommand{\bfrho}{{\boldsymbol{\rho}}}
\newcommand{\rbar}{{\bar\rho}}
\newcommand{\s}{\sigma}
\newcommand{\sbar}{{\bar\sigma}}
\renewcommand{\t}{\tau}
\newcommand{\z}{\zeta}

\newcommand{\D}{\Delta}
\newcommand{\G}{\Gamma}
\newcommand{\F}{\Phi}
\renewcommand{\L}{\Lambda}

\newcommand{\ga}{{\mathfrak{a}}}
\newcommand{\gb}{{\mathfrak{b}}}
\newcommand{\gI}{{\mathfrak{I}}}
\newcommand{\gM}{{\mathfrak{M}}}
\newcommand{\gn}{{\mathfrak{n}}}
\newcommand{\gp}{{\mathfrak{p}}}
\newcommand{\gP}{{\mathfrak{P}}}
\newcommand{\gq}{{\mathfrak{q}}}

\newcommand{\Abar}{{\bar A}}
\newcommand{\Ebar}{{\bar E}}
\newcommand{\kbar}{{\bar k}}
\newcommand{\Kbar}{{\overline K}}
\newcommand{\Pbar}{{\bar P}}
\newcommand{\Sbar}{{\bar S}}
\newcommand{\Tbar}{{\bar T}}
\newcommand{\gbar}{{\bar\gamma}}
\newcommand{\lbar}{{\bar\lambda}}
\newcommand{\ybar}{{\bar y}}
\newcommand{\phibar}{{\bar\f}}
\newcommand{\nubar}{{\overline\nu}}

\newcommand{\Acal}{{\mathcal A}}
\newcommand{\Bcal}{{\mathcal B}}
\newcommand{\Ccal}{{\mathcal C}}
\newcommand{\Dcal}{{\mathcal D}}
\newcommand{\Ecal}{{\mathcal E}}
\newcommand{\Fcal}{{\mathcal F}}
\newcommand{\Gcal}{{\mathcal G}}
\newcommand{\Hcal}{{\mathcal H}}
\newcommand{\Ical}{{\mathcal I}}
\newcommand{\Jcal}{{\mathcal J}}
\newcommand{\Kcal}{{\mathcal K}}
\newcommand{\Lcal}{{\mathcal L}}
\newcommand{\Mcal}{{\mathcal M}}
\newcommand{\Ncal}{{\mathcal N}}
\newcommand{\Ocal}{{\mathcal O}}
\newcommand{\Pcal}{{\mathcal P}}
\newcommand{\Qcal}{{\mathcal Q}}
\newcommand{\Rcal}{{\mathcal R}}
\newcommand{\Scal}{{\mathcal S}}
\newcommand{\Tcal}{{\mathcal T}}
\newcommand{\Ucal}{{\mathcal U}}
\newcommand{\Vcal}{{\mathcal V}}
\newcommand{\Wcal}{{\mathcal W}}
\newcommand{\Xcal}{{\mathcal X}}
\newcommand{\Ycal}{{\mathcal Y}}
\newcommand{\Zcal}{{\mathcal Z}}

\renewcommand{\AA}{\mathbb{A}}
\newcommand{\BB}{\mathbb{B}}
\newcommand{\CC}{\mathbb{C}}
\newcommand{\FF}{\mathbb{F}}
\newcommand{\GG}{\mathbb{G}}
\newcommand{\NN}{\mathbb{N}}
\newcommand{\PP}{\mathbb{P}}
\newcommand{\QQ}{\mathbb{Q}}
\newcommand{\RR}{\mathbb{R}}
\newcommand{\ZZ}{\mathbb{Z}}

\newcommand{\bfa}{{\boldsymbol a}}
\newcommand{\bfb}{{\boldsymbol b}}
\newcommand{\bfc}{{\boldsymbol c}}
\newcommand{\bfd}{{\boldsymbol d}}
\newcommand{\bfe}{{\boldsymbol e}}
\newcommand{\bff}{{\boldsymbol f}}
\newcommand{\bfg}{{\boldsymbol g}}
\newcommand{\bfi}{{\boldsymbol i}}
\newcommand{\bfj}{{\boldsymbol j}}
\newcommand{\bfk}{{\boldsymbol k}}
\newcommand{\bfm}{{\boldsymbol m}}
\newcommand{\bfn}{{\boldsymbol n}}
\newcommand{\bfp}{{\boldsymbol p}}
\newcommand{\bfr}{{\boldsymbol r}}
\newcommand{\bfs}{{\boldsymbol s}}
\newcommand{\bft}{{\boldsymbol t}}
\newcommand{\bfu}{{\boldsymbol u}}
\newcommand{\bfv}{{\boldsymbol v}}
\newcommand{\bfw}{{\boldsymbol w}}
\newcommand{\bfx}{{\boldsymbol x}}
\newcommand{\bfy}{{\boldsymbol y}}
\newcommand{\bfz}{{\boldsymbol z}}
\newcommand{\bfA}{{\boldsymbol A}}
\newcommand{\bfF}{{\boldsymbol F}}
\newcommand{\bfB}{{\boldsymbol B}}
\newcommand{\bfD}{{\boldsymbol D}}
\newcommand{\bfG}{{\boldsymbol G}}
\newcommand{\bfI}{{\boldsymbol I}}
\newcommand{\bfM}{{\boldsymbol M}}
\newcommand{\bfP}{{\boldsymbol P}}
\newcommand{\bfQ}{{\boldsymbol Q}}
\newcommand{\bfT}{{\boldsymbol T}}
\newcommand{\bfU}{{\boldsymbol U}}
\newcommand{\bfX}{{\boldsymbol X}}
\newcommand{\bfY}{{\boldsymbol Y}}
\newcommand{\bfzero}{{\boldsymbol{0}}}
\newcommand{\bfone}{{\boldsymbol{1}}}

\newcommand{\Aut}{\operatorname{Aut}}
\newcommand{\adj}{\operatorname{adj}}
\newcommand{\Berk}{{\textup{Berk}}}
\newcommand{\Birat}{\operatorname{Birat}}
\newcommand{\characteristic}{\operatorname{char}}
\newcommand{\codim}{\operatorname{codim}}
\newcommand{\Crit}{\operatorname{Crit}}
\newcommand{\crit}{{\textup{crit}}}
\newcommand{\critwt}{\operatorname{critwt}} 
\newcommand{\Cycle}{\operatorname{Cycles}}
\newcommand{\diag}{\operatorname{diag}}
\newcommand{\dimEnd}{{M}}  
\newcommand{\Disc}{\operatorname{Disc}}
\newcommand{\Div}{\operatorname{Div}}
\newcommand{\Df}{{Df}}  
\newcommand{\Dom}{\operatorname{Dom}}
\newcommand{\dyn}{{\textup{dyn}}}
\newcommand{\End}{\operatorname{End}}
\newcommand{\PortEndPt}{{\textup{endpt}}} 
\newcommand{\END}{\smash[t]{\overline{\operatorname{End}}}\vphantom{E}}
\newcommand{\EndPoint}{E}  
\newcommand{\ExtOrbit}{\mathcal{EO}} 
\newcommand{\Fbar}{{\bar{F}}}
\newcommand{\Fix}{\operatorname{Fix}}
\newcommand{\Fiber}{\operatorname{Fiber}}
\newcommand{\Fit}{\operatorname{Fit}}
\newcommand{\FOD}{\operatorname{FOD}}
\newcommand{\FOM}{\operatorname{FOM}}
\newcommand{\Frame}{\operatorname{Fr}}
\newcommand{\Gal}{\operatorname{Gal}}
\newcommand{\GITQuot}{/\!/}
\newcommand{\GL}{\operatorname{GL}}
\newcommand{\GR}{\operatorname{\mathcal{G\!R}}}
\newcommand{\Hom}{\operatorname{Hom}}
\newcommand{\Index}{\operatorname{Index}}
\newcommand{\Image}{\operatorname{Image}}
\newcommand{\Isom}{\operatorname{Isom}}
\newcommand{\hhat}{{\hat h}}
\newcommand{\Ker}{{\operatorname{ker}}}
\newcommand{\Ksep}{K^{\text{sep}}}  
\newcommand{\Length}{\operatorname{Length}}
\newcommand{\Lift}{\operatorname{Lift}}
\newcommand{\limstar}{\lim\nolimits^*}
\newcommand{\limstarn}{\lim_{\hidewidth n\to\infty\hidewidth}{\!}^*{\,}}
\newcommand{\Mat}{\operatorname{Mat}}
\newcommand{\maxplus}{\operatornamewithlimits{\textup{max}^{\scriptscriptstyle+}}}
\newcommand{\MOD}[1]{~(\textup{mod}~#1)}
\newcommand{\Model}{\operatorname{Model}}
\newcommand{\Mor}{\operatorname{Mor}}
\newcommand{\Moduli}{\mathcal{M}}
\newcommand{\MODULI}{\overline{\mathcal{M}}}
\newcommand{\Norm}{{\operatorname{\mathsf{N}}}}
\newcommand{\notdivide}{\nmid}
\newcommand{\normalsubgroup}{\triangleleft}
\newcommand{\NS}{\operatorname{NS}}
\newcommand{\onto}{\twoheadrightarrow}
\newcommand{\ord}{\operatorname{ord}}
\newcommand{\Orbit}{\mathcal{O}}
\newcommand{\Pcase}[3]{\par\noindent\framebox{$\boldsymbol{\Pcal_{#1,#2}}$}\enspace\ignorespaces}
\newcommand{\pd}{p}       
\newcommand{\bfpd}{\bfp}  
\newcommand{\Per}{\operatorname{Per}}
\newcommand{\Perp}{\operatorname{Perp}}
\newcommand{\PrePer}{\operatorname{PrePer}}
\newcommand{\PGL}{\operatorname{PGL}}
\newcommand{\Pic}{\operatorname{Pic}}
\newcommand{\Portrait}{\mathfrak{Port}}  
\newcommand{\prim}{\textup{prim}}
\newcommand{\Prob}{\operatorname{Prob}}
\newcommand{\Proj}{\operatorname{Proj}}
\newcommand{\Qbar}{{\bar{\QQ}}}
\newcommand{\rank}{\operatorname{rank}}
\newcommand{\Rat}{\operatorname{Rat}}
\newcommand{\reduced}{{\textup{red}}}
\newcommand{\Resultant}{\operatorname{Res}}
\newcommand{\Residue}{\operatorname{Residue}} 
\renewcommand{\setminus}{\smallsetminus}
\newcommand{\sgn}{\operatorname{sgn}}
\newcommand{\shafdim}{\operatorname{ShafDim}}
\newcommand{\SL}{\operatorname{SL}}
\newcommand{\Span}{\operatorname{Span}}
\newcommand{\Spec}{\operatorname{Spec}}
\renewcommand{\ss}{{\textup{ss}}}
\newcommand{\stab}{{\textup{stab}}}
\newcommand{\Stab}{\operatorname{Stab}}
\newcommand{\SemiStable}[1]{\textup{(SS$_{#1}$)}}  
\newcommand{\Stable}[1]{\textup{(St$_{#1}$)}}      
\newcommand{\Support}{\operatorname{Support}}
\newcommand{\Sym}{\operatorname{Sym}}  
\newcommand{\TableLoopSpacing}{{\vrule height 15pt depth 10pt width 0pt}} 
\newcommand{\tors}{{\textup{tors}}}
\newcommand{\Trace}{\operatorname{Trace}}
\newcommand{\trianglebin}{\mathbin{\triangle}} 
\newcommand{\tr}{{\textup{tr}}} 
\newcommand{\UHP}{{\mathfrak{h}}}    
\newcommand{\val}{\operatorname{val}} 
\newcommand{\wt}{\operatorname{wt}} 
\renewcommand{\>}{\rangle}
\newcommand{\Hilb}{\operatorname{Hilb}} 
\newcommand{\id}{\operatorname{id}} 
\newcommand{\pr}{\operatorname{pr}} 

\newcommand{\pmodintext}[1]{~\textup{(mod}~#1\textup{)}}
\newcommand{\ds}{\displaystyle}
\newcommand{\longhookrightarrow}{\lhook\joinrel\longrightarrow}
\newcommand{\longonto}{\relbar\joinrel\twoheadrightarrow}
\newcommand{\SmallMatrix}[1]{%
  \left(\begin{smallmatrix} #1 \end{smallmatrix}\right)}


\begin{abstract}
We study arithmetic distribution relations and the inverse function
theorem in algebraic and arithmetic geometry, with an emphasis on
versions that can be applied uniformly across families of varieties
and maps. In particular, we prove two explicit versions of the inverse
function theorem, the first via general distribution and separation
inequalities that may be of independent interest, the second via a
careful implementation of classical Newton iteration.
\end{abstract}

\maketitle

\tableofcontents

\section{Introduction}
\label{section:introduction}

In this article, we study arithmetic distribution relations and the
inverse function theorem in algebraic and arithmetic geometry, with an
emphasis on versions that can be applied uniformly across families of
varieties and maps. Roughly speaking, such results have the following
general form:

\begin{theoremtemplate}[Distribution Relation]
Let~$\f:W\to{V}$ be a map between metric spaces. Assume that~$V,W,\f$
satisfy suitable hypotheses. Then for all~$P\in{W}$ and all~$q\in{V}$,
\[
d_V\bigl( \f(P),q \bigr) \gg\ll \prod_{Q\in\f^{-1}(q)} d_W(P,Q)^{e_\f(Q)}
\]
for appropriately defined local multiplicities~$e_\f(Q)$.
\end{theoremtemplate}

\begin{theoremtemplate}[Inverse Function Theorem]
Let~$\f:W\to{V}$ be a map between metric spaces. Assume that~$V,W,\f$
satisfy suitable hypotheses. There are exponents~$n,m>0$ and a
subset~$\Rcal_\f\subset{W}$ so that if~$P\in{W}$ and~$q\in{V}$ satisfy
\[
d_V\bigl(\f(P),q\bigr) \ll d_V(P,\Rcal_\f)^n,
\]
then there is a \textup(unique\textup)~$Q\in{W}$ satisfying
\[
\f(Q)=q
\quad\text{and}\quad
d_W(P,Q) \ll d_V\bigl(\f(P),q\bigr) \cdot d_V(P,\Rcal_\f)^{-m}.
\]
\end{theoremtemplate}

In both of these formulations, the implicit constants may depend
on~$V,W,\f$, and also possibly on the distance of~$(P,q)$ to some sort
of ``boundary'' of~$W\times{V}$.

\begin{remark}
In this paper, we use ``arithmetic distance'' $\d$ instead of usual distance $d$ 
since it is compatible with height theory in arithmetic geometry.
Roughly speaking, the two ``distance functions'' $\d$ and $d$ are related each other 
via the formula $-\log d(\cdot, \cdot) = \d (\cdot, \cdot)$.
\end{remark}

We are interested in the case that~$V$ and~$W$ are algebraic varieties
defined over either a global field or a complete field
and~$\f:W\to{V}$ is a generically finite morphism.  In this setting,
versions of the distribution relation and inverse function theorem
were given in~\cite{silverman:arithdistfunctions}, where the inverse
function theorem was proven by combining the distribution relation
with a separation result. Unfortunately, the proof of the stated
distribution relation in~\cite{silverman:arithdistfunctions} is
incorrect due to a miscalculation with multiplicities.

The primary goals of this article may thus be summarized as follows:
\begin{parts}
\Part{\textbullet}
Give a counterexample to the arithmetic distribution relation
in~\cite[Proposition~6.2(b)]{silverman:arithdistfunctions} and explain
where the error in the proof occurs.  See
Section~\ref{section:counterexample}.
\Part{\textbullet}
Prove that the distribution relation
in~\cite{silverman:arithdistfunctions} is correct if $\dim{V}=1$ or
if~$\f$ is {\'e}tale.  See
Section~\ref{section:distribution-separation-inversefunction}.
\Part{\textbullet}
Prove that the distribution relation is valid in general as a~$\ll$
inequality.  Using the distribution inequality and a separation
estimate, prove the inverse function theorem as
in~\cite{silverman:arithdistfunctions} with the ramification divisor
replaced by a certain annihilator subscheme of~$W$.  For
maps~$\f:W\to{V}$ of degree~$d$, this gives the inverse function
theorem with exponents~$(n,m)=(d,d-1)$.  See
Section~\ref{section:distribution-separation-inversefunction}.
\Part{\textbullet}
Indicate briefly why the~$\dim{V}=1$ and~$\f$ {\'e}tale cases of the
inverse function theorem suffice for various arithmetic applications,
such as bounding the number of quasi-integral points on elliptic
curves~\cite{MR895285} and in orbits of maps
on~$\PP^1$~\cite{silverman:dynamicalintegerpoints}.  See
Section~\ref{section:quasiintergral}.
\Part{\textbullet}
Give an alternative independent proof of the inverse function theorem
over complete fields with optimal exponents~$(n,m)=(2,1)$.  We do
everything uniformly on quasi-projective varieties, so the proof, via
an {\'e}tale map reduction to~$\AA^N$ and then classical Newton
iteration, requires a lot of finicky computation, as well as requiring
handling the non-archimedean and archimedean cases separately.  See
Sections~\ref{section:invfunccompletefields},~\ref{section:newtonnonarch},~\ref{section:newtonarch}.
\end{parts}

\section{Notation and Terminology}

We mostly follow the terminology and notation
from~\cite{lang:diophantinegeometry}
and~\cite{silverman:arithdistfunctions}, including in particular the
following.

\begin{parts}
\Part{\textbullet}
A \emph{variety} over a field~$k$ is an irreducible and reduced scheme
of finite type over~$k$.
\Part{\textbullet}
For two closed subschemes $X, Y \subset V$ of a scheme $V$, the sum $X + Y$ is the closed subscheme which is defined by
$\Ical_{X}\Ical_{Y}$ where $\Ical_{X}$ and $\Ical_{Y}$ are ideal sheaves of $X$ and $Y$.
\Part{\textbullet}
In
Sections~\ref{section:counterexample}--\ref{section:distribution-separation-inversefunction}
we let~$K$ be a field with a complete set of inequivalent absolute
values~$M_K$ that are normalized so that the product formula holds.
We fix an algebraic closure~$\Kbar$ for~$K$ and write~$M(\Kbar)$ for
the set of absolute values on~$\Kbar$ extending the
absolute values in~$M_K$.  In
Sections~\ref{section:controots}--\ref{section:newtonarch}, we let~$K$
be a field that is complete with respect to an absolute
value~\text{$|\,\cdot\,|$}.
\Part{\textbullet}
An~\emph{$M_K$-constant} is a function~$\g:{M(\Kbar)}\to\RR_{\geq 0}$ such that~$\g(v)$
depends only on the restriction~$v|_K$ and
\text{$\bigl\{v|_K:\g(v)\ne0\big\}$} is a finite subset of~$M_K$.
\Part{\textbullet}
Let~$V/K$ be a variety. A function $\mu:V(\Kbar)\times{M(\Kbar)}\to\RR$
is \emph{$M_K$-bounded} if there exists an~$M_K$-constant~$\g$
satisfying~$\mu(P,v)\le\g(v)$ for all
$(P,v)\in{V}(\Kbar)\times{M(\Kbar)}$.
\Part{\textbullet}
We write~$O(M_K)$
to indicate relations that hold up to an~$M_K$-constant. For
example, we use the following notation, where the right-hand property
is required to hold for some~$C>0$ and some $M_K$-constant~$\g$.
\begin{align*}
  f=g+O(M_K)  &\;\Longleftrightarrow\;  |f-g|\le\g. \\
  f\le g+O(h)+O(M_K)  &\;\Longleftrightarrow\;  f\le g+C|h|+\g. \\
  f =  g+O(h)+O(M_K)  &\;\Longleftrightarrow\;  |f-g| \le C|h|+\g. \\
  f \ll g + O(M_K)    &\;\Longleftrightarrow\;  f \le Cg + \g.\\
  f \gg\ll g + O(M_K)    &\;\Longleftrightarrow\;  f \ll g + O(M_K) \\
  &\hspace*{4em}~\text{and}~g \ll f + O(M_K).
\end{align*}
\end{parts}

\subsection{Bounded sets}
\label{section:boundedsets}
In this section we recall some standard definitions regarding bounded subsets
of schemes defined over valued fields. For further material on bounded
sets, see for example~\cite[Ch.~10, Sec.~1]{lang:diophantinegeometry}.

\begin{definition}
Let~$\bigl(K,|\,\cdot\,|\bigr)$ be a complete field.  Let~$U=\Spec{A}$ be an
affine scheme of finite type over~$K$.  A subset~$B \subset{U}(K)$ is
called \emph{bounded} or \emph{affine bounded} if for every~$f\in{A}$,
we have
\[
\sup_{x\in B} {}  \bigl|f(x)\bigr|  <\infty.
\]
\end{definition}

\begin{remark}
Let~$\f \colon U \longrightarrow U'$ be a morphism between affine
schemes of finite type over~$K$.
\begin{parts}
  \Part{(1)}
  If~$B \subset U(K)$ is a bounded subset of~$U$, then
  $f(B)\subset{U'}(K)$ is a bounded subset of~$U'$.
  \Part{(2)}
  If~$\f$ is finite and~$B'\subset{U'}(K)$ is a bounded subset
  of~$U'$, then the set~$\f^{-1}(B') = \{x \in U(K) : \f(x) \in B'\}$ is a bounded subset of~$U$.
\end{parts}
\end{remark}

\begin{definition}
Let~$\bigl(K,|\,\cdot\,|\bigr)$ be a complete field.  Let~$X$ be a
scheme of finite type over~$K$.  A subset~$B \subset X(K)$ is called
\emph{bounded} if there is a finite open affine cover
$\{U_{i}\}_{i=1}^{r}$ of~$X$ and affine bounded subsets
$B_{i}\subset{U}_{i}(K)$ such that~$B=\bigcup_{i=1}^{r}B_{i}$.
\end{definition}

\begin{definition}
Let~$\bigl(K,|\,\cdot\,|\bigr)$ be a complete field.
Let~$X =\Spec A$ be an affine scheme of finite type over~$K$.
A \emph{standard bounded subset} is a subset~$B \subset X(K)$  of the form
\[
B = \{ x \in X(K) : \bigl|f_{1}(x)\bigr| \leq b_{1}, \dots , \bigl|f_{r}(x)\bigr| \leq b_{r}\},
\]
where~$A=K[f_1,\ldots,f_r]$ and~$b_1,\ldots,b_r>0$.  Note that a
standard bounded subset is an affine bounded subset.
\end{definition}

\begin{remark}
More generally, for a field~$K$ with a set of absolute values~$M_K$,
one says that a subset~$X\subset\Spec(A)(K)\times{M_K}$
is~\emph{$M_K$-bounded} if for every~$f\in{A}$ there is
an~$M_K$-constant~$\g_f$ such that
\[
\sup_{\substack{x\in\Spec(A)(K)\\ \text{s.t. $(x,v)\in X$}\\}}  \bigl|f(x)\bigr|_v \le e^{\g_f(v)}.
\]
\end{remark}

\subsection{Local heights attached to subschemes and arithmetic distance functions}
\label{section:htsandarithdistfunc} 

We briefly recall from~\cite{silverman:arithdistfunctions} the
notation and construction of local height functions attached to closed
subschemes, arithmetic distance functions, and local height functions
attached to the boundary of a quasi-projective scheme.  We refer the
reader to~\cite{silverman:arithdistfunctions} for further details.
However, we note that most of the existing literature,
including~\cite{silverman:arithdistfunctions}, assumes that the base
scheme is irreducible and reduced, but in this paper we at times use
base schemes that lack these properties.  We refer the reader
to~\cite{htsclosedsubschemes2020} for an extension of the theory of
local heights to schemes that are not necessarily irreducible or
reduced.

Let~$V/K$ be a projective variety. We can assign to proper each closed
subscheme~$X\subset{V}$ a local height function
\[
\l_X : V(\Kbar)\times M_K \longrightarrow \RR\cup\infty,
\]
uniquely determined up to an~$M_K$-bounded function by the properties
that if~$X=D$ is an effective divisor, then~$\l_X=\l_D$ is the usual
Weil local height, and if~$X$ and~$Y$ are closed subschemes, then
$\l_{X\cap{Y}}=\min\{\l_X,\l_Y\}$. (The intersection~$X\cap{Y}$ is
defined to be the scheme whose ideal sheaf
is~$\Ical_{X\cap{Y}}=\Ical_X+\Ical_Y$.)  These subscheme local heights
have a number of natural functorial properties, as described
in~\cite[Theorem~2.1]{silverman:arithdistfunctions}, including
functoriality~$\l_{\f^*X}=\l_X\circ\f$ for morphisms~$\f$.

Local height function $\l_{X}$ is bounded below up to $M_{K}$-bounded function, 
so we can always assume it is non-negative if the difference by $M_{K}$-bounded function does not matter. 

Let~$\D(V)\subset{V}\times{V}$ be the diagonal.
The \emph{arithmetic distance function} on~$V$ is the local height
\[
\d_V = \l_{\D(V)}.
\]
It is well-defined up to an~$M_{K}$-bounded function, and satisfies a
number of standard properties described
in~\cite[Proposition~3.1]{silverman:arithdistfunctions}, including the
following two triangle inequalities, where we omit~$v\in{M_K}$ from
the notation:
\begin{align*}
  \d_V(P,R) &\ge \min\bigl\{ \d_V(P,Q),\,\d_V(Q,R) \bigr\}. \\
  \l_X(Q) &\ge \min\bigl\{\l_X(P),\,\d_V(P,Q)\bigr\}.
\end{align*}

If~$V$ is merely quasi-projective, we embed~$V$ in a projective
variety~$\overline{V}$ and define the boundary~$\partial{V}$ to
equal~$\overline{V}\setminus{V}$ with its induced-reduced scheme
structure. Then to each closed subscheme~$X\subset{V}$ we can assign a
local height function~$\l_X$ that is well-defined up
to~$O(\l_{\partial{V}})$. These quasi-projective local height
functions inherit the functorial properties of the projective local
heights, except that every relation holds only up
to~$O(\l_{\partial{V}})$.

\begin{remark}
We also take this opportunity to correct some typographical errors
in~\cite{silverman:arithdistfunctions}. In the definition of the
union~$X\cup{Y}$ in~(iii) on page~196
of~\cite{silverman:arithdistfunctions}, the ideal should be
denoted~$\Ical_{X\cup{Y}}$. And
in~\cite[Theorem~2.1(e)]{silverman:arithdistfunctions}, the local
height should be for the union, not the intersection, so the displayed
formula in the statement of~(e) on page~198 and the proof of~(e) on
page~199 should read
\[
\max\{\l_X,\l_Y\} \le \l_{X\cup Y} \le \l_X+\l_Y.
\]
\end{remark}

\section{A counterexample to the arithmetic distribution relation in~\cite{silverman:arithdistfunctions}}
\label{section:counterexample}

The arithmetic distribution relation as stated
in~\cite{silverman:arithdistfunctions} says the following:

\begin{proposition}
\label{proposition:distrel1}
\textup{(\cite[Proposition~6.2(b)]{silverman:arithdistfunctions})}
Let~$\f:W\to V$ be a finite map of smooth quasi-projective varieties.
Let~$P\in{W}$ and~$q\in{V}$. Then
\[
\delta_{V}(\f(P), q; v) 
= \sum_{\substack{Q \in W(\Kbar)\\ \f(Q)=q\\}}
e_{\f}(Q) \delta_{W}(P, Q;  v) + O( \lambda_{\partial (W \times V)}(P,q; v)).
\]
Here~$e_\f(Q)$ is the ramification index of~$\f$ at~$Q$, so for
example, if~$Q$ is not in the ramification locus of~$\f$, then
$e_\f(Q/q)=1$.
\end{proposition}

\begin{example}
\label{example:distrelcounterexample1}
We give a simple counter-example to
Proposition~\ref{proposition:distrel1}. We consider the map
\[
\begin{CD}
  W = \PP^1\times\PP^1 @>\f>> \PP^1\times\PP^1 = V \\
  \bigl([x,y],[z,w]\bigr) @>>> \bigl([x^2,y^2],[z,w]\bigr).\\
\end{CD}
\]
We fix an absolute value~$v$ on~$K$, and for notational convenience,
we drop~$v$ from the notation. We take
\[
q = \bigl([0,1],[0,1]\bigr) \quad\text{and}\quad
P = \bigl([a,1],[b,1]\bigr)\quad\text{with $|a|<1$ and $|b|<1$.}
\]
Then
\[
\f^{-1}(q)=\{Q\}=\bigl\{ \bigl([0,1],[0,1]\bigr) \bigr\}
\quad\text{with $e_\f(Q/q)=2$,}
\]
and
\[
\f(P) = \bigl([a^2,1],[b,1]\bigr).
\]
Under our assumption that ~$|a|<1$ and~$|b|<1$, we see that
\begin{align}
\label{eqn:dVfPqab}
\d_V(\f (P),q) &= -\log\max\bigl\{|a^2|,|b|\bigr\},\\
\label{eqn:dVfPqa2b}
\d_W(P,Q) &= -\log\max\bigl\{|a|,|b|\bigr\}.
\end{align}
Proposition~\ref{proposition:distrel1} says that
\begin{equation}
\label{eqn:dVfPqeq2dWPQO1}
\d_V(\f (P),q) = 2\d_W(P,Q) + O(1),
\end{equation}
but by varying~$a$ and~$b$, we see from~\eqref{eqn:dVfPqab}
and~\eqref{eqn:dVfPqa2b} that the right-hand side
of~\eqref{eqn:dVfPqeq2dWPQO1} may be strictly larger than the
left-hand side. Indeed, for any fixed ~$1\le\kappa\le2$, we find that
\[
\d_V(\f (P),q)/\d_W(P,Q) \longrightarrow \kappa
\quad\text{as~$|b|=|a|^\kappa\to0$.}
\]
\end{example}

\begin{example}
\label{example:distrelcounterexample2}
We use Example~\ref{example:distrelcounterexample1} to indicate where
the proof of Proposition~\ref{proposition:distrel1}
in~\cite{silverman:arithdistfunctions} goes wrong.  For this example
only, we use notation from~\cite{silverman:arithdistfunctions}, which
may be slightly different from the notation used in the rest of this
paper. (See also Remark~\ref{remark:strictcontainment} for another
similar example.)
\par
The proof in~\cite{silverman:arithdistfunctions} begins with the case
that~$\f:W\to{V}$ is a Galois cover, say with Galois group
$\Aut(W/V)=\{\t_1,\ldots,\t_n\}$. It is then asserted that
\begin{equation}
  \label{eqn:ffpullbackdiagonal}
  (\f\times\f)^*\D(V) = \sum_{i=1}^n (1\times\t_i)^*\Delta(W).
\end{equation}
This equality of schemes is not correct.  Thus in
Example~\ref{example:distrelcounterexample1}, writing~$\gI(X)$ for the
ideal sheaf of a scheme~$X$, we have
\begin{align*}
  \gI \bigl( (\f\times\f)^*\D(V) \bigr) &= \bigl( (x_1y_2)^2-(y_2x_1)^2, z_1w_2-z_2w_1 \bigr), \\
  \gI \bigl( \D(W) + (1\times\t)^*\D(W) \bigr) &= (x_1y_2-y_2x_1, z_1w_2-z_2w_1)\\
  &\hspace{3em} {}\cdot(x_1y_2+y_2x_1, z_1w_2-z_2w_1).
\end{align*}
This gives an inclusion of ideals,
\[
\gI \bigl( \D(W) + \t^*\D(W) \bigr)
\subset
\gI \bigl( (\f\times\f)^*\D(V) \bigr),
\]
but the ideals are not equal, since for example the right-hand ideal
contains~$z_1w_2-z_2w_1$ and the left-hand ideal does not. Of course,
the underlying reduced schemes are the same, since
\[
\sqrt{\gI \bigl( \D(W) + \t^*\D(W) \bigr)}
=
\gI \bigl( (\f\times\f)^*\D(V) \bigr).
\]
In general, there is a
distribution inequality; see
Section~\ref{section:distributioninequality}.
\end{example}

\begin{remark}
\label{remark:ADRcorrect}
As shown by Example~\ref{example:distrelcounterexample1}, the
arithmetic distance relation described in
Proposition~\ref{proposition:distrel1} is incorrect in the stated
generality. There are, however, two important cases for which the
proof of Proposition~\ref{proposition:distrel1}
in~\cite{silverman:arithdistfunctions} is correct, and thus for which
the distribution relation and its application to the inverse function
theorem are valid.
\begin{parts}
  \Part{\textbullet}
  The varieties~$V$ and~$W$ are smooth and the map~$\f:W\to{V}$ is
  {\'e}tale. In this case there is no ramification, so the relevant
  ideal sheaves are automatically reduced
  \Part{\textbullet}
  The varieties~$V$ and~$W$ are smooth of dimension~$1$, in which case
  the ramification divisor is~$0$-dimensional.
\end{parts}
See Lemma~\ref{lem:curveetalecases} for the subscheme formula that is
key to proving the distribution relation in these two cases.
\end{remark}

\section{Applications to quasi-integral points}
\label{section:quasiintergral} 
Various quantitative versions of the inverse function theorem have
been used by the second
author~\cite{MR2782672,MR895285,silverman:dynamicalintegerpoints} to
study integral points.  In this section we briefly indicate the
relevance of the present paper to these earlier results.

The paper~\cite{MR895285} proves uniform height estimates for
quasi-$S$-integral points on families of elliptic curves, and more
generally on families of abelian varieties. The proof uses a method of
Siegel that involves taking the pull-back of the multiplication-by-$m$
map.  For the application in~\cite{MR895285}, one looks at an abelian
scheme, i.e., a family of abelian varieties~$A\to{T}$ over a not
necessarily complete base variety~$T$, and applies the inverse
function theorem to the multiplication-by-$m$ map~$[m]:A\to{A}$. The
map~$[m]$ is {\'e}tale, so as noted in Remark~\ref{remark:ADRcorrect},
the proofs of the distribution relation and the inverse function
theorem in~\cite{silverman:arithdistfunctions} are correct,
so~\cite{MR895285} does not require the present paper.

The paper~\cite{silverman:dynamicalintegerpoints} proves an analogue
of Siegel's integral point theorem for~$f$-orbits of points
in~$\PP^1$, where~$f\in\Qbar(z)$ is a rational map of degree at
least~$2$. A key step in the proof uses Siegel's pull-back idea, but
in this case the inverse function theorem is applied to an
iterate~$f^{\circ{n}}$ of~$f$. The map~$f^{\circ{n}}:\PP^1\to\PP^1$ is
highly ramified, but the inverse function theorem on~$\PP^1$ is much
easier than the general case, and a self-contained proof of the
required theorem is given in~\cite{silverman:dynamicalintegerpoints}.
Uniform versions of the results
in~\cite{silverman:dynamicalintegerpoints} were given by Hsia and the
second author~\cite{MR2782672}. The proof includes an application of
the inverse function theorem from~\cite{silverman:arithdistfunctions}
to the highly ramified iterates of a family of rational maps
\text{$f:\PP^1_T\to\PP^1_T$} over a base variety~$T$. Thus the inverse
function theorem references in~\cite{MR2782672} should be replaced by
references to the present paper.\footnote{We remark that it was noted
  explicitly in \cite[Section~3]{MR2782672} that ``it is undoubtedly
  possible to give a direct, albeit long and messy, proof of the
  desired [inverse function theorem] result.'' The proof of the
  inverse function theorem via Newton iteration in
  Section~\ref{section:invfunccompletefields} may be viewed as such a
  long, messy, and direct proof}

\section{Distribution, separation, and the inverse function theorem}
\label{section:distribution-separation-inversefunction}

In this section, we prove some estimates on arithmetic distance
functions involving inverse images by a finite morphisms.  The core
inequality is the distribution inequality describe in
Theorem~\ref{thm:distribution-inequality}.  We combine it with the
separation inequality in Proposition~\ref{prop:separation} to prove a
quantitative multivariable inverse function theorem; see
Theorem~\ref{thm:glinvfunc}.  This section thus provides a correction
to \cite[\S6]{silverman:arithdistfunctions}.

\subsection{The distribution inequality/relation}
\label{section:distributioninequality}

\begin{definition}
Let~$\f \colon W \longrightarrow V$ be a finite flat morphism between schemes of finite type over a field~$k$.
Let~$k'$ be an algebraically closed field containing~$k$.
For~$x \in W(k')$, define the \emph{multiplicity} of~$\f$ at~$x$ by
\[
e_{\f}(x) = {\rm length}_{\Ocal_{W_{k'}, x}} \Ocal_{W_{k'}, x}/\f^{*} \mathfrak{m}_{\f(x)}\Ocal_{W_{k'}, x} 
\] 
where~$W_{k'}=W\times_{\Spec{k}}\Spec{k'}$ and~$V_{k'}=V\times_{\Spec{k}}\Spec{k'}$, where~$x$ and~$\f(x)$ are closed points
of~$W_{k'}$ and~$V_{k'}$, respectively, and where~$\mathfrak{m}_{\f(x)}$ is
the maximal ideal of~$\Ocal_{V_{k'},\f(x)}$.  Note that
\[
e_{\f}(x) = \dim_{k'} \Ocal_{W_{k'}, x}/\f^{*} \mathfrak{m}_{\f(x)}\Ocal_{W_{k'}, x},
\]
since~$k'$ is algebraically closed.
If~$\f$ has constant degree~$d$, then for any~$y \in V(k')$, we have
\[
\sum_{x \in W(k'), \f(x) = y} e_{\f}(x) = d.
\]
\end{definition}

\begin{theorem}[Distribution inequality]\label{thm:distribution-inequality}
Let~$\f \colon W \longrightarrow V$ be a generically \'etale finite
flat morphism between quasi-projective geometrically integral
varieties over~$K$.
\begin{parts}
\Part{(1)}
For all~$(P, q, v) \in W(\Kbar) \times V(\Kbar) \times M(\Kbar)$, we have
\begin{align}
\label{eqn:distinequality}
  \delta_{V}&(\f(P),  q; v)  \notag\\
  & \leq \sum_{\substack{Q \in W(\Kbar)\\ \f(Q)=q\\}}
  e_\f(Q) \delta_{W}(P, Q; v) + O\bigl( \lambda_{\partial (W \times V)}(P,q;v)\bigr) + O(M_K).
\end{align}
\Part{(2)}
Suppose~$V$ and~$W$ are smooth.  Then~\eqref{eqn:distinequality} is an
equality in each of the following situations\textup:
\begin{parts}
  \Part{\textbullet}
  $\dim V= \dim W =1$.
  \Part{\textbullet}
  $\f$ is \'etale.
\end{parts}
\end{parts}
\end{theorem}

\subsubsection{Inequalities between closed subschemes}

We prove some containments between closed subschemes, from which we
deduce Theorem~\ref{thm:distribution-inequality}.  This can be done
over an arbitrary field~$k$, so we let~$\f\colon{W}\longrightarrow{V}$
be a generically \'etale finite flat morphism between quasi-projective
geometrically integral varieties over~$k$.

Let~$d=\deg\f$.
Since~$\f$ is flat, we see that
\[
W \xrightarrow{\;1\times\f\;} W\times V \xrightarrow{\;\;\operatorname{proj}_2\;\;} V
\]
is a flat family of zero-dimensional closed subschemes of~$W$ of
length~$d$.  Thus it defines a
morphism~$V\to\Hilb^{d}(W)$.  Let~$\Phi$ denote the
composite~$V\to\Hilb^{d}(W)\to W^{(d)}$, where
$W^{(d)}=W^{d}/S_{d}$ is the~$d$-times symmetric power of~$W$ and
at a geometric point~$y\in{V}$, we
have~$\Phi(y)=\{x_{1},\dots,x_{d}\}$, where $x_{1},\dots,x_{d}$ are
the points in the inverse image of~$y$ by~$\f$, listed with
multiplicity.
The second morphism is the Hilbert-Chow morphism. See Figure~\ref{figure:fg1}

\begin{figure}
\[
\xymatrix@C=60pt{
W \ar[d]_{(\id, \f)} & & \\
W \times V \ar[d]_{\pr_{2}} \ar[r] & W \times \Hilb^{d}(W) \ar[d] & \\
V \ar[r] \ar@/_20pt/[rr]_{\Phi} & \Hilb^{d}(W) \ar[r]^(.55){\text{Hilbert-Chow}} & W^{(d)}
}
\]
\caption{The map $\Phi$ that inverts $\f$}
\label{figure:fg1}
\end{figure}

We form the fiber product
\[
\xymatrix{
Z = V \times_{W^{(d)}} W^{d}  \ar[r]^(.7){\psi} \ar[d]_{p} & W^{d} \ar[d]^{\pi} \\
V \ar[r]_{\Phi} &  W^{(d)}
}
\]
where~$\pi$ is the quotient morphism.  We remark that~$Z$ need not be
either reduced or irreducible.

\begin{lemma}\label{lem:univopen}
The morphisms~$\pi$ and~$\f$ are universally open, i.e., every base
change of~$\pi$ and~$\f$ is an open map.
\end{lemma}
\begin{proof}
Since~$\f$ is a flat morphism of finite type between Noetherian
schemes, it is universally open.

Note that~$\pi$ is finite surjective, and in particular it is
universally submersive, i.e., every base change~$\pi'\colon{X}\to{Y}$
of~$\pi$ is surjective and the topology on~$Y$ is the quotient
topology of~$X$.  Let~$U\subset{X}$ be an open subset.
Then~$\pi'^{-1}(\pi'(U))=S_{d}\cdot{U}$, since~$S_{d}$ acts
transitively on every geometric fiber of~$\pi'$.
Since~$S_{d}\cdot{U}$ is open, submersivity of~$\pi'$
implies~$\pi'(U)$ is also open.
\end{proof}

We consider the diagram in Figure~\ref{figure:fg2},
where~$\mu_{i}=(\id\times\pr_{i})\circ(\id\times\psi)$, and where all
of the squares are cartesian.

\begin{figure}
\[
\xymatrix{
&Z \ar@{^{(}->}[rd]_(.30){(\pr_{i}\circ \psi, \id)} \ar[r]^(.30){\sim} & \mu_{i}^{-1}( \Delta_{W}) \ar@{^{(}->}[d] \ar[rr]&& \Delta_{W} \ar@{^{(}->}[d] \\
T \ar[d]_{f} \ar@{^{(}->}[rr]_{ \iota} \ar@/_20pt/[dd]_{h} && W \times Z  \ar[d]^{\f \times \id} \ar[r]^{\id \times \psi} & W \times W^{d} \ar[r]^{\id \times \pr_{i}} & W \times W\\
Z \ar@{^{(}->}[rr]_{(p, \id)} \ar[d]_{g} && V \times Z \ar[d]^{\id \times p} & &\\
\Delta_{V} \ar@{^{(}->}[rr] && V \times V &&
}
\]
\caption{Fitting together all of the spaces and maps}
\label{figure:fg2}
\end{figure}

\begin{lemma}\
\begin{enumerate}\label{lem:firstpropdiag}
\item
  \label{lem:firstpropdiag:support}
  $\Support( T )= \Support \Bigl( \sum_{i=1}^{d} \mu_{i}^{-1}( \Delta_{W}) \Bigr)$.
\item
  \label{lem:firstpropdiag:R0}
  Every irreducible component of~$T$ dominates~$ \Delta_{V}$ via~$h$
  and~$T$ is generically reduced, i.e., satisfies the condition \textup{(R0)}.
\item
  \label{lem:firstpropdiag:dominate}
  For every irreducible component~$E$ of~$W\times{Z}$, we
  have~$(\f\times{p})(E)=V\times{V}$ and~$\mu_{i}(E)=W\times{W}$ as
  sets.
\item
  \label{lem:firstpropdiag:intersection}
  If~$(x,z)\in(\f\times{p})^{-1}(\Delta_{V})(\kbar)$ is contained in
  exactly~$r$ of the sets
  \[
  \mu_{1}^{-1}( \Delta_{W})(\kbar),\, \dots,\, \mu_{d}^{-1}( \Delta_{W})(\kbar),
  \]
  then~$e_{\f}(x) =r$.
\end{enumerate}
\end{lemma}

\begin{proof}
\eqref{lem:firstpropdiag:support}\enspace This follows from the
construction.  \par\noindent\eqref{lem:firstpropdiag:R0}\enspace By
Lemma~\ref{lem:univopen},~$h$ is an open map, and therefore every
irreducible component of~$T$ dominates~$ \Delta_{V}$.
Let~$U\subset{V}$ be a dense open subset over which~$\f$ is \'etale.
Then $\f\times{p}$ is \'etale over~$U\times{U}$.  This implies~$h$ is
\'etale over the dense open subset~$\Delta_{U}\subset\Delta_{V}$,
and hence~$T$ is generically reduced.
\par\noindent\eqref{lem:firstpropdiag:dominate}\enspace
By Lemma~\ref{lem:univopen}, the map~$\f\times{p}$ is an open map. It
is also finite, and hence we get~$(\f\times{p})(E)=V\times{V}$.  The
last statement follows from this and the following commutative
diagram:
\[
\xymatrix{
W\times Z \ar[r]^{\mu_{i}} \ar[d]_{\f \times p} & W \times W \ar[ld]^{\f \times \f} \\
V \times V &.
}
\]
\par\noindent\eqref{lem:firstpropdiag:intersection}\enspace
This follows from the construction.
\end{proof}

We next consider the diagram given in Figure~\ref{figure:fg3}, where
we view all of the schemes as schemes over the~$Z$ at the bottom of
the diagram.  Let~$j\colon{Z}_{\reduced}\to{Z}$ be the reduced scheme.
Base change of the diagram in Figure~\ref{figure:fg3} along the
morphism~$j$ gives the diagram in Figure~\ref{figure:fg4}, where the
subscripts~$(-)_{Z_{\reduced}}$ stands for base change
to~$Z_{\reduced}$.

\begin{figure}[ht]
 \[
\xymatrix{
&Z \ar@{^{(}->}[rd]_(.30){(\pr_{i}\circ \psi, \id)} \ar[r]^(.30){\sim} & \mu_{i}^{-1}( \Delta_{W}) \ar@{^{(}->}[d] \\
T \ar[d]_{f} \ar@{^{(}->}[rr]_{ \iota} && W \times Z  \ar[d]^{\f \times \id}\\
Z \ar@{^{(}->}[rr]_{(p, \id)}  && V \times Z \ar[d]^{\pr_{2}}\\
&& Z
}
\]
\caption{Another commutative diagram}
\label{figure:fg3}
\end{figure}

\begin{figure}[ht]
\begin{align}\label{diagram:reduction}
\xymatrix{
&Z_{\reduced} \ar@{^{(}->}[rd]_(.30){(\pr_{i}\circ \psi \circ j, \id)} \ar[r]^(.30){\sim} & \mu_{i}^{-1}( \Delta_{W})_{Z_{\reduced}} \ar@{^{(}->}[d] \\
T_{Z_{\reduced}} \ar[d]_{f_{Z_{\reduced}}} \ar@{^{(}->}[rr]_{ \iota_{Z_{\reduced}}} && W \times Z_{\reduced}  \ar[d]^{\f \times \id}\\
Z_{\reduced} \ar@{^{(}->}[rr]_{(p\circ j, \id)}  && V \times Z_{\reduced} \ar[d]^{\pr_{2}}\\
&& Z_{\reduced}
}
\end{align}
\caption{Base change to $Z_{\reduced}$}
\label{figure:fg4}
\end{figure}

\begin{lemma}\label{lem:generalcontainment}
The scheme~$T_{Z_{\reduced}}$ is reduced, and in particular, we have
\[
T_{Z_{\reduced}} \subset \sum_{i=1}^{d} \mu_{i}^{-1}( \Delta_{W})_{Z_{\reduced}}
\]
as closed subschemes of~$W \times Z_{\reduced}$.
\end{lemma}

\begin{proof}
Lemma~\ref{lem:firstpropdiag}\eqref{lem:firstpropdiag:support} tells
us that these two closed subschemes have the same the support, so it
is enough to show that~$T_{Z_{\reduced}}$ is reduced.

First, note that~$T_{Z_{\reduced}}\to{T}$ is a thickening of schemes.
Thus by Lemma~\ref{lem:firstpropdiag}\eqref{lem:firstpropdiag:R0}, the
scheme~$T_{Z_{\reduced}}~$ satisfies~(R0).  By the
diagram~\eqref{diagram:reduction}, the map~$f_{Z_{\reduced}}$ is the
base change of~$\f\times\id$ and therefore it is flat finite.
Since~$Z_{\reduced}$ is reduced, it satisfies the condition~(S1), and
therefore~$T_{Z_{\reduced}}$ also satisfies~(S1).
Thus~$T_{Z_{\reduced}}$ is~(R0) and~(S1), which is equivalent to being
reduced.
\end{proof}

In general,~$Z$ is not reduced, and the
containment~$T\subset\sum_{i=1}^{d}\mu_{i}^{-1}( \Delta_{W})$ is not
true.  However, with some additional assumptions, we can show that
these closed subschemes are equal.

\begin{lemma}
\label{lem:curveetalecases}
\begin{parts}
  \Part{(1)}
  If~$\f \times p$ is flat, then~$W \times Z$ and~$T$ are reduced.
  \Part{(2)}
 If~$\dim W= \dim V=1$ and~$W, V$ are smooth, then~$\f \times p$ is flat and 
\[
T=\sum_{i=1}^{d}\mu_{i}^{-1}( \Delta_{W}).
\]
\Part{(3)}
 If~$V$ is smooth and~$\f$ is \'etale, then~$\f \times p$ is \'etale and 
\[
T=\sum_{i=1}^{d}\mu_{i}^{-1}( \Delta_{W}).
\]
\end{parts}
\end{lemma}

\begin{proof}
(1)\enspace
By the same argument in the proof of
Lemma~\ref{lem:firstpropdiag}~\eqref{lem:firstpropdiag:R0},~$W\times{Z}$
and~$T$ are (R0).  By the assumption,~$W\times{Z}$ and~$T$ are flat
over~$V\times{V}$ and~$ \Delta_{V}$ respectively.  This implies they
are (S1) and we are done.
\par\noindent(2)\enspace
Since~$W$ is a smooth curve, the symmetric power~$W^{(d)}$ is smooth
and~$\pi\colon{W}^{d}\to{W}^{(d)}$ is flat.  This
implies~$p\colon{Z}\to{V}$ is flat and therefore~$\f\times{p}$ is
flat.  Since~$V \times V$ is a smooth surface, it is in particular
Cohen-Macaulay, and the scheme~$W \times Z$, which is flat finite
over~$V\times{V}$, is also Cohen-Macaulay.  By the reducedness
of~$W\times{Z}$ and
Lemma~\ref{lem:firstpropdiag}~\eqref{lem:firstpropdiag:dominate},
$\mu_{i}^{-1}(\Delta_{W})\subset{W}\times{Z}$ are effective Cartier
divisors.  Thus~$\sum_{i=1}^{d} \mu_{i}^{-1}( \Delta_{W})$ is also an
effective Cartier divisor and in particular it is (S1).  Since
$\mu_{i}^{-1}(\Delta_{W})\simeq{Z}$ are reduced and any two of them do
not have common irreducible components
(cf.\ Lemma~\ref{lem:firstpropdiag}~\eqref{lem:firstpropdiag:intersection}),
$\sum_{i=1}^{d} \mu_{i}^{-1}( \Delta_{W})$ is generically reduced.
This proves~$\sum_{i=1}^{d} \mu_{i}^{-1}( \Delta_{W})$ is reduced and
we are done.
\par\noindent(3)\enspace
If~$\f$ is \'etale,~$\pi$ is \'etale at every point over~$ \Phi(V)$.
Thus~$p$ is \'etale, and therefore~$\f \times p$ is \'etale.  In
particular,~$\mu_{i}^{-1}(\Delta_{W})\simeq{Z}$ are reduced.  Since
$\f$ is \'etale,~$\mu_{i}^{-1}( \Delta_{W})$ are disjoint
(cf.\ Lemma~\ref{lem:firstpropdiag}~\eqref{lem:firstpropdiag:intersection})
and therefore~$\sum_{i=1}^{d} \mu_{i}^{-1}( \Delta_{W})$ is reduced
and we are done.
\end{proof}

\subsubsection{Proof of Distribution inequality}

\begin{proof}[Proof of Theorem $\ref{thm:distribution-inequality}$]
We consider the diagram in Figure~\ref{figure:fg5}, where
\text{$\mu_{i}=(\id\times\pr_{i})\circ(\id\times\psi)$} and all the
squares are cartesian.  Note that \text{$W\times{Z}_{\reduced}$} is
reduced since~$W$ is geometrically integral.  By
Lemma~\ref{lem:firstpropdiag}\eqref{lem:firstpropdiag:dominate},
$(\id\times{j})\circ(\f\times{p})$ maps all associated points
of~$W\times{Z}_{\reduced}$ to the generic point of~$V\times{V}$,
and~$\mu_{i}$ maps them to the generic point of~$W\times{W}$.
Therefore we get
\begin{align*}
  \delta_{V}  \circ (\f \times (p\circ j))
  &=( \lambda_{ \Delta_{V}} + O( \lambda_{\partial (V \times V)})) \circ (\f \times (p\circ j)) \\
  & = \lambda_{T_{Z_{\reduced}}} + O( \lambda_{\partial (W \times Z_{\reduced})})
\end{align*}
and
\begin{align*}
\sum_{i=1}^{d} \delta_{W} \circ \mu_{i} \circ (\id \times j)
&= \sum_{i=1}^{d} (\lambda_{ \Delta_{W}}+O( \lambda_{\partial (W\times W)}))\circ \mu_{i} \circ (\id \times j)\\
&=\sum_{i=1}^{d} \lambda_{\mu_{i}^{-1}( \Delta_{W})_{Z_{\reduced}}} + O( \lambda_{\partial (W \times Z_{\reduced})})\\
&= \lambda_{\sum_{i=1}^{d} \mu_{i}^{-1}( \Delta_{W})_{Z_{\reduced}}} + O( \lambda_{\partial (W \times Z_{\reduced})}).
\end{align*}

\begin{figure}[t]
 \[
\xymatrix{
&\mu_{i}^{-1}( \Delta_{W})_{Z_{\reduced}} \ar@{^{(}->}[d] \ar[rd]& & & \\
T_{Z_{\reduced}} \ar@{^{(}->}[r] \ar[rd] & W \times Z_{\reduced} \ar[rd]_{\id \times j} & \mu_{i}^{-1}( \Delta_{W}) \ar@{^{(}->}[d] \ar[rr] && \Delta_{W} \ar@{^{(}->}[d] \\
&T \ar[d]_{h} \ar@{^{(}->}[r] & W \times Z  \ar[d]^{\f \times p} \ar[r]^{\id \times \psi} & W \times W^{d} \ar[r]^{\id \times \pr_{i}} & W \times W \\
&\Delta_{V} \ar@{^{(}->}[r]& V \times V 
}
\]
\caption{Yet another commutative diagram}
\label{figure:fg5}
\end{figure}

Lemma~\ref{lem:generalcontainment} tells us that
\[
\lambda_{T_{Z_{\reduced}}}  \leq \lambda_{\sum_{i=1}^{d} \mu_{i}^{-1}( \Delta_{W})_{Z_{\reduced}}} + O( \lambda_{\partial (W \times Z_{\reduced})}).
\]
This implies 
\[
\delta_{V}  \circ (\f \times (p\circ j))  \leq \sum_{i=1}^{d} \delta_{W} \circ \mu_{i} \circ (\id \times j) + O( \lambda_{\partial (W \times Z_{\reduced})}).
\]
Here note that~$p\circ j$ is a finite morphism, and in particular it
is a proper morphism, so we have
\[
\lambda_{\partial (W \times Z_{\reduced})} \gg \ll \lambda_{\partial (W \times V)} \circ (\id \times (p\circ j)).
\] 
Thus we get
\begin{equation}
  \label{eqn:dvfpjledw}
\delta_{V} \circ (\f \times (p\circ j)) \leq \sum_{i=1}^{d} \delta_{W}
\circ \mu_{i} \circ (\id \times j) + O(\lambda_{\partial (W \times V)}
\circ (\id \times (p\circ j))).
\end{equation}
Now take arbitrary~$(P, q, v) \in W(\Kbar) \times V(\Kbar) \times
M(\Kbar)$.  Take a point~$z \in Z_{\reduced}(\Kbar)$ such
that~$p\circ j(z) = q$.  Then the definition of~$Z$ tells us
that~$\psi\circ j(z)$ is a point of the form
\[
(Q_{1},\dots , Q_{d}) \in W(\Kbar)^{d}
\]
where~$Q_{i}$'s are preimages of~$q$ by~$\f$ counted with
multiplicity.  Plugging the point~$(P, z)$ into~\eqref{eqn:dvfpjledw},
we get
\begin{align*}
\delta_{V}(\f(P),q;v ) &\leq \sum_{i=1}^{d} \delta_{W}(P, Q_{i};v) + O( \lambda_{\partial (W \times V)}(P, q;v))\\
&= \sum_{\substack{Q \in W(\Kbar)\\\f(Q)=q\\}} e_{\f}(Q) \delta_{W}(P, Q; v) + O( \lambda_{\partial (W \times V)}(P,q;v)).
\end{align*}
\par
For the second statement, by Lemma~\ref{lem:curveetalecases}, if~$W,V$
are smooth, and if either they have dimension~$1$ or if~$\f$ is
\'etale, then we have the equality
\[
\lambda_{T_{Z_{\reduced}}}  = \lambda_{\sum_{i=1}^{d} \mu_{i}^{-1}( \Delta_{W})_{Z_{\reduced}}} + O( \lambda_{\partial (W \times Z_{\reduced})}).
\]
Thus by the same argument, we get the desired equality.
\end{proof}

\subsection{The separation inequality}
\label{section:separationinequality}

In this section we quantify the assertion that distinct inverse image
points cannot be too close to one another.

\begin{proposition}[Separation]\label{prop:separation}
Let~$W$ and~$V$ be quasi-projective varieties over~$K$, and let
$\f\colon{W}\to{V}$ be a generically \'etale, generically finite
morphism.  Let~$\Ical$ be the annihilator ideal sheaf
of~$\Omega_{W/V}$.  Then for~$v\in M(\Kbar)$ and for all
points~$Q,Q'\in{W}(\Kbar)$ such that~$\f(Q)=\f(Q')$ and~$Q\neq{Q'}$\textup:
\begin{parts}
\Part{(1)}
In general we have
\[
\d_{W}(Q,Q';v) \leq \l_{\Ical}(Q;v) + O( \lambda_{\partial (W \times W)}(Q,Q';v)) + O(M_K).
\]
\Part{(2)}
If~$\f$ is proper, then 
\[
\d_{W}(Q,Q';v) \leq \l_{\Ical}(Q;v) + O( \lambda_{\partial V}(\f(Q);v)) + O(M_K).
\]
\Part{(3)}
Let~$\Fit_{0}(\Omega_{W/V})$ be the~$0$-th Fitting ideal
of~$\Omega_{W/V}$. Then~\textup{(1)} and~\textup{(2)} are true
with~$\l_{\Ical}$ replaced by~$\l_{\Fit_{0}(\Omega_{W/V})}$.
\Part{(4)}
Assume that~$V$ and~$W$ are smooth and that~$\f$ is finite,
and let~$R(\f)$ be the ramification divisor of~$\f$.
Then~\textup{(1)} and~\textup{(2)} are true
with~$\l_{\Ical}$ replaced by~$\l_{R(\f)}$.
\end{parts}
\end{proposition}

\begin{lemma}
\label{lem:separation-ideal}
Let~$k$ be a field.  Let~$W, V$ be varieties over~$k$ and let
\text{$\f\colon{W}\to{V}$} be a generically \'etale generically finite
morphism.  Let~$\Ical$ be the annihilator ideal sheaf
of~$\Omega_{W/V}$.  Then
\[
(\pr_{1}^{*}\Ical) \Ical_{\D(W)} + \Ical_{\D(W)}^{2} \subset (\f \times \f)^{*}\Ical_{\D(V)} + \Ical_{\D(W)}^{2}
\]
on~$W \times W$.
In particular, since the~$0$-th Fitting ideal sheaf~$ \Fit_{0}(\Omega_{W/V})$ is contained in~$\Ical$, we have
\[
(\pr_{1}^{*}\Fit_{0}(\Omega_{W/V})) \Ical_{\D(W)} + \Ical_{\D(W)}^{2} \subset (\f \times \f)^{*}\Ical_{\D(V)} + \Ical_{\D(W)}^{2}.
\]
Moreover, if~$W,V$ are smooth and~$\f$ is finite, then
\begin{align}
(\pr_{1}^{*}\Ical_{R(\f)}) \Ical_{\D(W)} + \Ical_{\D(W)}^{2} \subset (\f \times \f)^{*}\Ical_{\D(V)} + \Ical_{\D(W)}^{2} \label{separation-ramif}
\end{align}
where~$\Ical_{R(\f)}$ is the ideal sheaf of the ramification divisor~$R(\f)$.
\end{lemma}
\begin{proof}
Consider the exact sequence
\begin{equation}
  \label{eqn:fstarOmega}
  \f^{*} \Omega_{V/k} \longrightarrow \Omega_{W/k} \longrightarrow \Omega_{W/V} \longrightarrow 0.
\end{equation}
Since the sheaf of differentials is the conormal sheaf of the diagonal, we can rewrite the sequence~\eqref{eqn:fstarOmega} as
\[
(\f \times \f)^{*} \Ical_{\D(V)}/\Ical_{\D(W)} (\f \times
\f)^{*}\Ical_{\D(V)} \longrightarrow \Ical_{\D(W)} / \Ical_{\D(W)}^{2}
\longrightarrow \Omega_{W/V} \longrightarrow 0.
\]
Then using the fact that~$\Ical\Omega_{W/V}=0$, we get
\[
(\pr_{1}^{*}\Ical) \Ical_{\D(W)} + \Ical_{\D(W)}^{2} \subset (\f
\times \f)^{*}\Ical_{\D(V)}\Ocal_{W\times W} + \Ical_{\D(W)}^{2}.
\]
The gives everything except for the last assertion of the lemma, which
follows from the fact that if~$W$ and~$V$ are smooth and~$\f$ is
finite, then~$\Fit_{0}(\Omega_{W/V})=\Ical_{R(\f)}$.
\end{proof}

\begin{remark}
Let~$\Ical$ is the annihilator ideal sheaf of~$\Omega_{W/V}$.
Then if~$W$ and~$V$ are smooth and~$\f$ is finite, we have
\[
\Ical_{R(\f)} = \Fit_{0}(\Omega_{W/V}) \subset \Ical.
\]
Further, if we let~$V\bigl(\Fit_{0}(\Omega_{W/V})\bigr)$ denote the closed subset
defined by the ideal sheaf~$\Fit_{0}(\Omega_{W/V})$, then we also have
\[
V\bigl(\Fit_{0}(\Omega_{W/V})\bigr) = \Support \Omega_{W/V}.
\]
Thus if~$R(\f)$ is reduced, then~$\Ical_{R(\f)}=\Ical$.
\end{remark}

\begin{remark}
When~$W$ and~$V$ are smooth curves, it follows easily from the proof
of Lemma~\ref{lem:separation-ideal} that we have an equality
\[
(\pr_{1}^{*}\Ical_{R(\f)}) \Ical_{\D(W)} + \Ical_{\D(W)}^{2} = (\f
\times \f)^{*}\Ical_{\D(V)} + \Ical_{\D(W)}^{2}.
\]
\end{remark}

\begin{remark}
\label{remark:strictcontainment}  
The containment~\eqref{separation-ramif} in
Lemma~\ref{lem:separation-ideal} may be strict in general in dimension
greater than~$1$.  For example, let
\[
\f \colon \AA^{2} \longrightarrow \AA^{2}, \quad (x,y) \mapsto (x^{2}+xy+y^{2}, xy + 1).
\]
Then the defining function of the ramification divisor is~$2(x^{2}-y^{2})$, and
 \begin{align}
   (\pr_{1}^{*}\Ical_{R(\f)}) & \Ical_{\D(W)} + \Ical_{\D(W)}^{2} \notag\\
   &=  (x^{2}-y^{2})(x-z, y-w) + (x-z, y-w)^{2},  \label{eqn:ideal1}\\
   (\f \times \f)^{*} & \Ical_{\D(V)} + \Ical_{\D(W)}^{2}  \notag\\
   &= \bigl( (x^{2}+xy+y^{2}) - (z^{2}+zw+w^{2}), (xy+1)-(zw+1) \bigr) \notag\\
   &\hspace*{3em} +  (x-z, y-w)^{2}. \label{eqn:ideal2}
\end{align}
It is easy to see that~$xy-zw$ is in the ideal~\eqref{eqn:ideal2},
but that it is not in~\eqref{eqn:ideal1};
cf.\ Example~\ref{example:distrelcounterexample2}.
\end{remark}

\begin{proof}[Proof of Proposition $\ref{prop:separation}$]
The inclusion of ideal sheaves in Lemma~\ref{lem:separation-ideal}
translates into the inequality of local height functions  
\[
\min\{ \l_{\Ical} \circ \pr_{1}+\d_{W}, 2\d_{W}\} \geq \min\{ \d_{V} \circ (\f \times \f), 2\d_{W}\} - \Cl{sepx}\l_{ \partial (W \times W)} + O(M_K).
\]
For~$Q,Q'\in{W}(\Kbar)$ such that~$\f(Q)=\f(Q')$ and~$Q\neq{Q'}$, we have 
\[
\d_{V}((\f \times \f)(Q,Q') ;v) = \d_{V}(q,q) = \infty,
\]
and therefore
\begin{multline*}
\l_{\Ical}(Q;v) + \d_{W}(Q,Q';v) \\
\geq 2\d_{W}(Q,Q';v) -  \Cr{sepx}\l_{ \partial (W \times W)}(Q,Q';v) +O(M_K).
\end{multline*}
This implies
\[
\l_{\Ical}(Q;v) \geq \d_{W}(Q,Q';v) -  \Cr{sepx}\l_{ \partial (W \times W)}(Q,Q';v) +O(M_K),
\]
which proves the first statement.
\par
If we suppose that~$\f$ is proper, then we have
\[
\lambda_{\partial (W\times W)}  \gg \ll \lambda_{\partial (V \times V)} \circ (\f \times \f) + O(M_K),
\]
and we also always have
\[
\lambda_{\partial (V \times V)}  \gg \ll \lambda_{\partial V} \circ \pr_{1} + \lambda_{\partial V} \circ \pr_{2} + O(M_K).
\]
These imply that for all~$Q,Q'\in{W}(\Kbar)$ such that~$\f(Q)=\f(Q')$,
we have
\[
\lambda_{\partial (W\times W)}(Q,Q';v) \gg \ll \lambda_{\partial V}(\f(Q);v) + O(M_K).
\]
Thus we get the second statement
\end{proof}

\subsection{The inverse function theorem: Version I} 

In this section we prove a version of the inverse function theorem.
The statement and proof are modeled after the statement and proof
in~\cite{silverman:arithdistfunctions}, but corrected by the use of
the alternative ramification local height coming from the corrected
versions of the distribution and separation lemmas proven in
Sections~\ref{section:distributioninequality}
and~\ref{section:separationinequality}.

\begin{theorem}[Inverse function theorem: Version I]
\label{thm:glinvfunc}
Fix/define the following quantities\textup:
\begin{parts}
  \Part{\textbullet}
  $V$ and~$W$ are quasi-projective geometrically integral varieties
  defined over~$K$.
  \Part{\textbullet}
   $\f\colon{W}\to{V}$ is a generically \'etale finite
  flat surjective morphism of degree~$d$ defined over~$K$.
  \Part{\textbullet}
  $\operatorname{Ann}(\Omega_{W/V})$ is the annihilator ideal sheaf of~$\Omega_{W/V}$.
  \Part{\textbullet}
  $A(\f)\subset{W}$ is the closed subscheme defined by~$\operatorname{Ann}(\Omega_{W/V})$.
  \Part{\textbullet}
  $\d_{W}$ and~$\d_{V}$ are arithmetic distance functions on~$W$ and~$V$.
  \Part{\textbullet}
   $\l_{A(\f)}$  is a local height function associated with~$A(\f)$.
  \Part{\textbullet}
  $\lambda_{\partial (W\times V)}$ is a local height boundary function
  for~$W \times V$.
\end{parts}
\begin{parts}
\Part{(a)}
There exist constants~$\Cl{1},\Cl{3} \in \RR_{>0}$ and
$M_{K}$-constants~$\Cl{2},\Cl{4}$ such that the following
holds\textup:
\par
If the triple~$(P,q,v)\in{W}(\Kbar)\times{V}(\Kbar)\times M(\Kbar)$ satisfies
\[
\d_{V}(\f(P),q;v) \geq d \l_{A(\f)}(P;v) + \Cr{1} \lambda_{\partial (W\times V)}(P,q;v) + \Cr{2}(v),
\]
then there exists a point~$Q\in{W}(\Kbar)$ satisfying
\begin{align*}
  \f(Q) = q&\quad\text{and} \\
\d_{W}(P,Q;v) \geq \d_{V}\bigl( \f(P),q;v\bigr) - (d & -1) \l_{A(\f)}(P;v)  \\
 & - \Cr{3} \lambda_{\partial (W\times V)}(P,q;v) - \Cr{4}(v).
\end{align*}
\Part{(b)}
If we take~$\Cr{2}$ to be an appropriate positive real number, instead
of an~$M_{K}$-constant, and if we also assume that~$P\notin{A}(\f)$,
then the point~$Q$ in~\textup{(a)} is unique.
\end{parts}
\end{theorem}

\begin{example}
The point~$Q$ may depend on the absolute value~$v$, as well as on~$P$
and~$q$. We illustrate with an example, which for convenience we write
using affine coordinates. We start with distinct non-archimedean
absolute values~$v$ and~$w$ that do not divide~$2$, and we
choose elements~$\a,\b\in{K}$ satisfying
\[
|\a|_v<1,\quad |\a|_w=1,\quad |\b|_v=1,\quad |\b|_w<1.
\]
We consider the map~$\f(x)=x^2$ and points
\[
P = 1,\quad Q = \frac{\a^n-\b^n}{\a^n+\b^n},\quad q = Q^2.
\]
The ramification divisor of~$\f$ is~$R(\f)=(0)$, so for any absolute value~$u$ we have
\[
\l_{R(\f)}(P;u) = \l_{(0)}(1;u) = -\log|1|_u = 0.
\]
We next determine the distance from~$\f(P)$ to~$q$. Taking~$u$ to
be~$v$ or~$w$, we compute
\begin{align*}
\d_{\PP^1}\bigl(\f(P),q;u\bigr)
= -\log| 1 - Q^2 |_u
&= -\log \left|\frac{4\a^n\b^n}{(\a^n+\b^n)^2} \right|_u \\
&= \begin{cases}
  -n\log|\a|_v &\text{if $u=v$,} \\
  -n\log|\b|_w &\text{if $u=w$.} \\
\end{cases}
\end{align*}
Hence if~$n$ is sufficiently large, then we have satisfied the
assumptions of the inverse function theorem for both~$v$ and~$w$, so
in each case there is a point in~$\f^{-1}(q)=\{\pm{Q}\}$ that is
appropriately close to~$P$. But for~$u\in\{v,w\}$, we find that
\begin{align*}
\d_{\PP^1}\bigl(P,\pm Q;u\bigr)
= -\log| -1 \pm Q |_u
& =-\log \left| -1 \pm \frac{\a^n-\b^n}{\a^n+\b^n} \right|_u \\
& = \begin{cases}
  -\log \left| \dfrac{2\a^n}{\a^n+\b^n} \right|_u &\text{if sign is $-$,} \\
  -\log \left| \dfrac{2\b^n}{\a^n+\b^n} \right|_u &\text{if sign is $+$,} \\
\end{cases} \\
& = \begin{cases}
  -n\log|\a|_v &\text{if $u=v$, sign is $-$,} \\
  -n\log|\b|_w &\text{if $u=w$, sign is $+$,} \\
  0 &\text{otherwise.} \\
\end{cases}
\end{align*}
Hence for the~$v$-adic absolute value, the inverse function theorem
requires us to take~$-Q$, while for the~$w$-adic absolute value we must
take~$Q$.
\end{example}

\begin{remark}
The morphism~$\f$ is \'etale outside~$A(\f)$.
In particular,~$e_{\f}(Q) = 1$ for all points~$Q \in (W \setminus A(\f))(\Kbar)$.
\end{remark}

\begin{remark}\label{rmk:annvsram}
Suppose that~$W$ and~$V$ are smooth.  Then in general, we
have~$A(\f)\subset{R}(\f)$ as closed subschemes.
Indeed, 
\[
\f^{*}\Omega_{V} \longrightarrow \Omega_{W} \longrightarrow \Omega_{W/V} \longrightarrow 0
\]
is a locally free resolution of $\Omega_{W/V}$ and the ideal of $R(\f)$ is locally generated by the determinant of the first map.
Hence
\[
\l_{A(\f)} \leq \l_{R(\f)} + O( \lambda_{\partial W}) + O(M_K),
\]
so if~$W$ and~$V$ are smooth, then Theorem~\ref{thm:glinvfunc} is true
with~$\l_{R(\f)}$ in place of~$\l_{A(\f)}$.
\end{remark}

\begin{proof}[Proof of Theorem~$\ref{thm:glinvfunc}$]
We fix boundary functions~$ \lambda_{\partial W}$ and~$\lambda_{\partial V}$.
Since
\[
 \lambda_{\partial (W\times V)} \gg \ll \lambda_{\partial W} + \lambda_{\partial V}  + O(M_K),
 \]
it is enough to show the statement for~$ \lambda_{\partial (W\times V)} = \lambda_{\partial W}+ \lambda_{\partial V}$.
Note that we have 
\[
\lambda_{\partial W} \gg \ll \lambda_{\partial V} \circ \f + O(M_K),
\]
since~$\f$ is a proper morphism.

In the following, we write~$C_{i}$ for positive real constants and
$C_{i}(v)$ for~$M_{K}$-constants. These constants are allowed to
depend on the varieties~$W$ and~$V$, on the map~$\f$, and on our
choice of local height and distance functions
$\d_{W},\d_{V},\l_{A(\f)},\lambda_{\partial W},\lambda_{\partial V}$.

By Theorem~\ref{thm:distribution-inequality} and Proposition
\ref{prop:separation}, we have:
\par\noindent\textbullet\enspace\textbf{Distribution inequality}:
\begin{align*}
  \d_{V}(\f(P),q;v) \leq \smash[b]{ \sum_{Q \in W(\Kbar), \f(Q)=q}}    e_{\f}(Q)\d_{W}&(P,Q;v) \\
  & + \Cl{di1} \lambda_{\partial (W\times V)}(P,q;v) +\Cl{di2}(v)
\end{align*}
for all~$(P,q,v)\in{W}(\Kbar)\times{V}(\Kbar)\times M(\Kbar)$.

\par\noindent\textbullet\enspace\textbf{Separation inequality}:
\begin{align*}
\d_{W}(Q,Q';v) \leq \l_{A(\f)}(Q;v) + \Cl{si1} \lambda_{\partial V}(\f(Q);v) + \Cl{si2}(v)
\end{align*}
for all~$(Q,Q',v)\in{W}(\Kbar)\times{W}(\Kbar)\times M(\Kbar)$ such
that~$\f(Q)=\f(Q')$ and~$Q\neq{Q'}$.

Let us fix
arbitrary~$(P,q,v)\in{W}(\Kbar)\times{V}(\Kbar)\times M(\Kbar)$.
For~$Q,Q'\in{W}(\Kbar)$ such that~$\f(Q)=\f(Q')=q$ and~$Q\neq{Q'}$, by
the triangle inequality and the separation inequality, we have
\begin{align*}
\min&\{ \d_{W}(P,Q';v), \d_{W}(P, Q;v)\}\\
&\leq \d_{W}(Q,Q';v) + \Cl{tr} \lambda_{\partial (W^3)}(P,Q,Q';v) + \C(v)\\
&\leq \l_{A(\f)}(Q;v) + \Cr{si1} \lambda_{\partial V}(\f(Q);v) +  \Cr{tr} \lambda_{\partial (W^3)}(P,Q,Q';v) + \C(v)\\
& \leq  \l_{A(\f)}(Q;v) + \Cl{ts1} \lambda_{\partial V}(q;v) + \Cr{tr} \lambda_{\partial W}(P;v) + \Cl{ts2}(v),
\end{align*}
where
$\lambda_{\partial(W^3)}=\lambda_{\partial{W}}+\lambda_{\partial{W}}+\lambda_{\partial{W}}$.

Let~$Q\in{W}(\Kbar)$ be a point with~$\f(Q)=q$ such that
\[
\d_{W}(P,Q;v) = \max\bigl\{ \d_{W}(P,Q';v) : Q'\in W(\Kbar),\, \f(Q')=q\bigr\}.
\]
Then for every~$Q'\in{W}(\Kbar)$ such that~$\f(Q')=q$ and~$Q'\neq{Q}$,
we have
\begin{align}
\label{ineq:Q'}
\d_{W}(P,Q';v)
& = \min\bigl\{ \d_{W}(P,Q';v), \d_{W}(P, Q;v)\bigr\}  \notag\\
& \leq  \l_{A(\f)}(Q;v) + \Cr{ts1} \lambda_{\partial V}(q;v) + \Cr{tr} \lambda_{\partial W}(P;v) + \Cr{ts2}(v).
\end{align}

Bounding the right-hand side of the distribution inequality by~\eqref{ineq:Q'}, we get  
\begin{align*}
\d&_{V} (\f(P),q;v) \\*
 & \leq  e_{\f}(Q)\d_{W}(P,Q;v) \\*
&\hspace{.2em}{} + \bigl(d-e_{\f}(Q)\bigr)\left(\l_{A(\f)}(Q;v) + \Cr{ts1} \lambda_{\partial V}(q;v) + \Cr{tr} \lambda_{\partial W}(P;v) + \Cr{ts2}(v) \right)\\*
&\hspace{.2em}{} + \Cr{di1} \lambda_{\partial (W\times V)}(P,q;v) +\Cr{di2}(v)\\*
& \leq  e_{\f}(Q)\d_{W}(P,Q;v) +(d-e_{\f}(Q))\l_{A(\f)}(Q;v) \\*
&\hspace{.2em}{}+  \Cl{tsd1} \lambda_{\partial (W\times V)}(P,q;v) + \Cl{tsd2}(v).
\end{align*}
\par
In summary, we have proved that for all
\[
(P,q,v)\in{W}(\Kbar)\times{V}(\Kbar)\times M(\Kbar)
\]
there is a point a point~$Q \in W(\Kbar)$ satisfying~$\f(Q)=q$ such
that
\begin{align}\label{ineq:key}
\d_{V}  (\f(P),q;v) 
&\leq   e_{\f}(Q)\d_{W}(P,Q;v) +(d-e_{\f}(Q))\l_{A(\f)}(Q;v) \notag \\*
&\hspace{4em}{} +  \Cr{tsd1} \lambda_{\partial (W\times V)}(P,q;v) + \Cr{tsd2}(v)
\end{align}
and
\begin{align*}
\d_{W}(P,Q;v) = \max\bigl\{ \d_{W}(P,Q';v) : Q'\in W(\Kbar),\; \f(Q')=q\bigr\}.
\end{align*}

Next, we compare~$\l_{A(\f)}(Q;v)$ and~$\l_{A(\f)}(P;v)$.
By the triangle inequality, we have
\begin{multline*}
\min\{ \l_{A(\f)}(Q;v), \d_{W}(P,Q;v)\} \\
\leq \l_{A(\f)}(P;v) + \Cl{2tr1} \lambda_{\partial (W\times W)}(P,Q;v) + \Cl{2tr2}(v),
\end{multline*}
where
$\lambda_{\partial(W\times{W})}=\lambda_{\partial{W}}+\lambda_{\partial{W}}$.

\par\noindent\textbf{Case 1}.\enspace
Suppose  that
\begin{align*}
\d_{W}(P,Q;v) \leq  \l_{A(\f)}(P;v) + \Cr{2tr1} \lambda_{\partial (W\times W)}(P,Q;v) + \Cr{2tr2}(v).
\end{align*}
By the distribution inequality and the choice of~$Q$, we have
\begin{align*}
 d\d_{W} & (P,Q;v) +  \Cr{di1} \lambda_{\partial (W\times V)}(P,q;v) +\Cr{di2}(v) \\
& \geq  \sum_{Q' \in W(\Kbar), \f(Q')=q}e_{\f}(Q')\d_{W}(P,Q';v) +  \Cr{di1} \lambda_{\partial (W\times V)}(P,q;v) +\Cr{di2}(v) \\
& \geq \d_{V}(\f(P),q;v).
\end{align*}
Thus we get  
\begin{align*}
\d_{V}(\f(P),q;v) 
&\leq d \left(\l_{A(\f)}(P;v) + \Cr{2tr1} \lambda_{\partial (W\times W)}(P,Q;v) + \Cr{2tr2}(v) \right) \\
& \omit\hfill${}+  \Cr{di1} \lambda_{\partial (W\times V)}(P,q;v) +\Cr{di2}(v)$\\
&\leq  d \l_{A(\f)}(P;v) + \Cl{dr1} \lambda_{\partial (W\times V)}(P,q;v) + \Cl{dr2}(v).
\end{align*}
Hence if we assume that
\begin{equation}
  \label{eqn:dVfPqvgeqdl}
  \d_{V}(\f(P),q;v) \geq d \l_{A(\f)}(P;v) + \Cr{1} \lambda_{\partial (W\times V)}(P,q;v) + \Cr{2}(v),
\end{equation}
then we get
\begin{multline*}
d \l_{A(\f)}(P;v) + \Cr{1} \lambda_{\partial (W\times V)}(P,q;v) + \Cr{2}(v)\\
 \leq  d \l_{A(\f)}(P;v) + \Cr{dr1} \lambda_{\partial (W\times V)}(P,q;v) + \Cr{dr2}(v).
\end{multline*}
Now if~$P\in{A}(\f)$, or equivalently if~$\l_{A(\f)}(P;v)=\infty$,
then~\eqref{eqn:dVfPqvgeqdl} implies that~$q=\f(P)$, so we can simply
take $Q=P$.  Thus we may assume that~$P\notin{A}(\f)$, in which case we
get
\begin{align}\label{ineq:assvs}
\Cr{1} \lambda_{\partial (W\times V)}(P,q;v) + \Cr{2}(v) \leq 
 \Cr{dr1} \lambda_{\partial (W\times V)}(P,q;v) + \Cr{dr2}(v).
\end{align}
Thus if we take~$\Cr{1} > \Cr{dr1}$ and take~$\Cr{2}$ large enough as
an~$M_{K}$-constant, then for
any~$(P,q,v)\in{W}(\Kbar)\times{V}(\Kbar)\times M(\Kbar)$, either
the inequality~\eqref{ineq:assvs} does not hold, or else
\[
\lambda_{\partial (W\times V)}(P,q;v) = \Cr{2}(v) = \Cr{dr2}(v)=0. 
\]
In the latter case, we have
\begin{align*}
\d_{V}(\f(P),q;v) = d \l_{A(\f)}(P;v) 
\end{align*}
and
\begin{align*}
&\d_{W}(P,Q;v) \\
&=  \l_{A(\f)}(P;v) + \Cr{2tr1} \lambda_{\partial (W\times W)}(P,Q;v) + \Cr{2tr2}(v)\\
& = \d_{V}(\f(P),q;v) - (d-1)\l_{A(\f)}(P;v) + \Cr{2tr1} \lambda_{\partial (W\times W)}(P,Q;v) + \Cr{2tr2}(v)\\
& \geq \d_{V}(\f(P),q;v) - (d-1)\l_{A(\f)}(P;v) - \C  \lambda_{\partial (W\times V)}(P,q;v) - \C(v).
\end{align*}
This proves the first statement for Case~1.

We also note that that if we take~$\Cr{2}$ to be a large positive
constant, rather than an~$M_{K}$-constant, and if we also assume that
$P\notin{A}(\f)$, then Case~1 does not happen.

\par\noindent\textbf{Case 2}.\enspace
We are reduced to the case that for appropriate choices
for the constants~$\Cr{1}$ and~$\Cr{2}$, we may assume that
\begin{align}\label{ineq:disram}
 \l_{A(\f)}(Q;v) & \leq \l_{A(\f)}(P;v) + \Cr{2tr1} \lambda_{\partial (W\times W)}(P,Q;v) + \Cr{2tr2}(v)  \notag\\
 &\leq \l_{A(\f)}(P;v) + \Cl{cr1} \lambda_{\partial (W\times V)}(P,q;v) + \Cl{cr2}(v).
\end{align}
By the same reasoning as earlier, we may assume that~$P\notin{A}(\f)$,
or equivalently, that~$\l_{A(\f)}(P;v)<\infty$.
Then~\eqref{ineq:disram} tells us that~$Q\notin{A}(\f)$
and~$e_{\f}(Q)=1$.

Applying~\eqref{ineq:key} and~\eqref{ineq:disram}, we get
\begin{align*}
&\d_{V}  (\f(P),q;v) \\*
&\leq   \d_{W}(P,Q;v) +(d-1)\left(   \l_{A(\f)}(P;v) + \Cr{cr1} \lambda_{\partial (W\times V)}(P,q;v) + \Cr{cr2}(v) \right) \\*
&\qquad{} +  \Cr{tsd1} \lambda_{\partial (W\times V)}(P,q;v) + \Cr{tsd2}(v)\\
&\leq  \d_{W}(P,Q;v) +(d-1)\l_{A(\f)}(P;v) +\Cl{fl1} \lambda_{\partial (W\times V)}(P,q;v) +\Cl{fl2}(v),
\end{align*}
which is what we want.  This completes the proof of~(a), i.e., the
existence of a point~$Q$ having the specified properties.
\par\noindent(b)\enspace It remains to show that if we take~$\Cr{2}$
to be a sufficiently large absolute constant, rather than
an~$M_K$-constant, then~$Q$ is unique.  Assume that~$P\notin{A}(\f)$.
We first choose~$\Cr{1},\dots,\Cr{4}$ so that~(a) holds, but we then
replace~$\Cr{2}$ with a large positive real number. As noted earlier,
this means that Case~1 in the proof of~(a) does not occur.

Suppose that there is a point~$Q'\in{W}(\Kbar)$ such that~$\f(Q')=q$
and~$Q'\neq{Q}$  and~$Q'$ satisfies
\begin{multline*}
\d_{V}(\f(P),q;v) \\
\leq \d_{W}(P,Q';v) +(d-1)\l_{A(\f)}(P;v) +\Cr{3} \lambda_{\partial (W\times V)}(P,q;v) +\Cr{4}(v).
\end{multline*}
This estimate and~\eqref{ineq:Q'} and~\eqref{ineq:disram} yield
\begin{align*}
d &\l_{A(\f)}(P;v) + \Cr{1} \lambda_{\partial (W\times V)}(P,q;v) + \Cr{2}\\
&\leq \d_{V}(\f(P),q;v)\\
& \leq \d_{W}(P,Q';v) +(d-1)\l_{A(\f)}(P;v) +\Cr{3} \lambda_{\partial (W\times V)}(P,q;v) +\Cr{4}(v)\\
& \leq d\l_{A(\f)}(P;v) + \Cl{contr1} \lambda_{\partial (W\times V)}(P,q;v) +\Cl{contr2}(v).
\end{align*}
If we replace~$\Cr{1}$ and~$\Cr{2}$ with larger constants that depend
on~$\Cr{contr1}$,~$\Cr{contr2}$, and $\lambda_{\partial(W\times{V})}$,
we obtain a contradiction.
\end{proof}

\begin{remark}
The proof of Theorem~\ref{thm:glinvfunc} shows that there is an
$M_{K}$-constant~$\g$ such that if
$\Cr{2}\colon M(\Kbar)\longrightarrow\RR_{\geq 0}$ satisfies the
strict inequality~$\Cr{2}(v)>\g(v)$ for all~$v\in M(\Kbar)$, and if
we assume that~$P\notin{A}(\f)$, then we still get the uniqueness
of~$Q$.
\end{remark}

\section{A version of continuity of roots}
\label{section:controots}
In Sections~\ref{section:controots}--\ref{section:newtonarch}, we
shift our focus to a field~$K$ that is complete with respect to a
fixed absolute value.  We thus let $\bigl(K,|\,\cdot\,|\bigr)$ be a
complete field, and since there is only one absolute value, we
drop~$v$ from the our notation for local heights and arithmetic
distance functions.

A theorem such as Theorem~\ref{thm:glinvfunc} may be viewed as a
quantitative higher dimensional variant of the classical theorem that
the roots of a univariate polynomial vary continuously with its
coefficients.  In our next result, we apply
Theorem~\ref{thm:glinvfunc} to a certain morphism to prove such a
result.  We use this later to prove a stronger inverse function
theorem; see
Theorem~\ref{thm:strong-inverse-function-all-absolute-values}.

\begin{proposition}
\label{prop:controots}
Let~$\bigl(K,|\,\cdot\,|\bigr)$ be a complete field with a non-trivial
absolute value \text{$|\,\cdot\,|$}.  Let~$D\in\RR_{>0}$ and
$n\in\ZZ_{>0}$.  Then there are constants
$\Cl{Kr1},\Cl{Kr2}\in\RR_{>0}$ such that the following holds.  Suppose
that\textup:
\begin{parts}
\Part{\textbullet} $f, g\in K[t]$ are monic polynomials of degree~$n$;
\Part{\textbullet} $|f| \leq D$ and~$|g| \leq D$, where~$|f|$ and~$|g|$ are the
  Gauss norms;\footnote{The Gauss norm of a polynomial is the maximum
    of the absolute values of its coefficients.}
  \Part{\textbullet} There is an ~$\a\in K$
  such that
\begin{equation}
\label{eqn:fa0fgckr1}
  f(\a)=0\quad \text{and}\quad |f-g| \leq e^{-\Cr{Kr1}} |f'(\a)|^{n}. 
\end{equation}
\end{parts}
Then there is~$\b \in K$ such that
\[
g(\b)=0 \quad\text{and}\quad |\a - \b||f'(\a)|^{n-1} \leq e^{\Cr{Kr2}} |f-g|.
\]
\end{proposition}

\begin{remark}
See Corollary~\ref{cor:stcontrootsnonarch}.  for a stronger version
Proposition~\ref{prop:controots}.  But when
\text{$\bigl(K,|\,\cdot\,|\bigr)$} is non-archimedean, we need
Proposition~\ref{prop:controots} to prove
Proposition~\ref{prop:newton-nonarhi}, to which we ultimately reduce
Corollary~\ref{cor:stcontrootsnonarch}.
\end{remark}

\begin{proof}[Proof of Proposition $\ref{prop:controots}$]
We consider the following morphism:
\begin{align*}
\f \colon \AA^{n+1}_{K} &\longrightarrow \AA^{n+1}_{K}\\
 (x_{0},\dots, x_{n-1}, t) &\longmapsto (x_{0},\dots, x_{n-1}, t^{n}+x_{n-1}t^{n-1}+\cdots + x_{1}t+x_{0}).
\end{align*}
The map~$\f$ is a generically \'etale, finite, surjective morphism of degree~$n$. Its ramification divisor is
\begin{align*}
R(\f) = \bigl( J:= nt^{n-1}+(n-1)x_{n-1}t^{n-2}+\cdots +x_{1}=0\bigr) \subset \AA^{n+1}_{K}.
\end{align*}

We fix an algebraic closure~$\Kbar$ of~$K$ and an extension of
\text{$|\,\cdot\,|$} to~$\Kbar$, which we also denote by
\text{$|\,\cdot\,|$}.  We consider the affine bounded subset
\begin{align*}
B := \{ \xi \in \AA^{n+1}(\Kbar) \mid  \| \xi \| \leq D \},
\end{align*}
where \text{$\|\,\cdot\,\|$} denotes the sup norm on the coordinates.
Since~$\f$ is finite,~$\f^{-1}(B) \subset \AA^{n+1}(\Kbar)$ is also an
affine bounded subset.  We apply Theorem~\ref{thm:glinvfunc} to
$\f_{\Kbar}$ on various bounded subsets, where we note that since we
are over a local field, the set~$M_K$ consists of a single absolute
value.  We may take
\begin{align*}
\d_{\AA^{n+1}}(\xi, \eta) &= \log \frac{1}{\|\xi -\eta \| }  && \text{on~$\bigl(\f^{-1}(B)\cup B\bigr)\times \bigl(\f^{-1}(B)\cup B\bigr)$,}\\[1\jot]
\l_{R(\f)}(\xi) &= \log \frac{1}{\bigl|J(\xi)\bigr|}  &&\text{on~$\f^{-1}(B)$,}\\[1\jot]
 \lambda_{\partial (\AA^{n+1} \times \AA^{n+1})}(\xi,\eta) &= 0  &&\text{on~$\f^{-1}(B) \times B$}.
\end{align*}
Theorem~\ref{thm:glinvfunc} and Remark~\ref{rmk:annvsram} tell us
that there are positive constants~$\Cr{Kr1},\Cr{Kr2}$ such that for
$\xi\in\f^{-1}(B)$ and~$\eta\in{B}$, if
\begin{align*}
\log \frac{1}{\| \f(\xi) - \eta \|} \geq n \log \frac{1}{|J(\xi)|} + \Cr{Kr1},
\end{align*}
then there is a~$\zeta \in \f^{-1}(B)$ satisfying
$\f(\zeta)=\eta$ and
\[
 \log \frac{1}{\| \xi - \zeta \|} + (n-1)\log \frac{1}{|J(\xi)|} + \Cr{Kr2} \geq  \log \frac{1}{\| \f(\xi) - \eta \|}.
\]

Rewriting this, we find that if~$\xi\in\f^{-1}(B)$ and~$\eta\in{B}$
satisfy
\begin{align*}
\| \f(\xi) - \eta\| \leq e^{-\Cr{Kr1}}|J(\xi)|^{n},
\end{align*}
then there is~$\zeta \in \f^{-1}(B)$ satisfying
\[
\f(\zeta)=\eta\quad\text{and}\quad
\|\xi - \zeta \| |J(\xi)|^{n-1} \leq e^{\Cr{Kr2}} \|\f(\xi) - \eta \|.
\]
We apply this to the points
\[
\xi = (a_{0},\dots, a_{n-1},\a)
\quad\text{and}\quad
\eta = (b_{0},\dots ,b_{n-1},0)
\]
associated to the polynomials
\[
f = t^{n}+a_{n-1}t^{n-1} + \cdots + a_{0}\\
\quad\text{and}\quad
g = t^{n}+b_{n-1}t^{n-1} + \cdots + b_{0}.
\]
It follows that if
\begin{align*}
|f-g| = \|\f(\xi) - \eta\| \leq e^{-\Cr{Kr1}} |J(\xi)|^{n} = e^{-\Cr{Kr1}} |f'(\a)|^{n},
\end{align*}
then there is a~$\b \in \Kbar$ such that~$g(\b)=0$ and
\[
 |\a -\b||f'(\a)|^{n-1} \leq \| \xi - (b_{0},\dots ,b_{n-1},\b) \| |f'(\a)|^{n-1} \leq e^{\Cr{Kr2}} |f-g|.
\]

It remains to show that~$\b \in K$, where we may need to increase the
value of~$\Cr{Kr1}$.  To this end, we may assume that~$f'(\a) \neq 0$,
since otherwise~\eqref{eqn:fa0fgckr1} tells us that~$f=g$, so we may
take~$\b=\a$.

Using the assumption that~$f'(\a)\ne0$, we can estimate
\begin{align}
  \label{eqn:gpbetagetfpalpha}
  \bigl|g'(\b)\bigr|
  &\ge \bigl|f'(\a)\bigr| - \bigl|f'(\a)-g'(\b)\bigr| \notag\\
  &\ge \bigl|f'(\a)\bigr| - \Cl{Kr3}|\a-\b|
  \quad\begin{tabular}[t]{@{}r@{}}for some $\Cr{Kr3}>0$ depending\\ on $n, D$, and $|\,\cdot\,|$.\\ \end{tabular} \notag\\[-3\jot]
  &\ge \bigl|f'(\a)\bigr| - \Cr{Kr3} e^{\Cr{Kr2}} \frac{|f-g|}{\bigl|f'(\a)\bigr|^{n-1}} \notag\\
  &\ge \bigl|f'(\a)\bigr| -   \Cr{Kr3}e^{\Cr{Kr2}}e^{-\Cr{Kr1}}\bigl|f'(\a)\bigr| \notag\\
  &= \Cl{Kr6}\bigl|f'(\a)\bigr|
  \quad\begin{tabular}[t]{@{}l@{}}
  where $\Cr{Kr6}>$ if we take an\\ appropriately large value for $\Cr{Kr1}$.\\ \end{tabular}
\end{align}
In particular, we note that~$g'(\b)\ne0$.  

Suppose~$\b \notin K$.  Let~$\b'\in\Kbar$ be
a~$\Gal(\Kbar/K)$-conjugate of~$\b$ with~$\b'\ne\b$.  Then
\begin{align*}
|\b'-\b| \geq \Cl{Kr4} |g'(\b)|  \geq \Cl{Kr5} |f'(\a)|
\end{align*}
where the first inequality is elementary and the second is~\eqref{eqn:gpbetagetfpalpha}.
Note that we also have 
\[
|\a - \b| \leq e^{\Cr{Kr2}-\Cr{Kr1}}|f'(\a)|,
\]
so if~$\Cr{Kr1}$ is large enough, then we get
\[
|\b'-\b| > |\a - \b|.
\]
It follows from Krasner's Lemma~\cite[Ch.~II, Sec.~2,
  Prop.~3]{lang:numbertheory} that~$\b\in{K}(\a)=K$.
\end{proof}

\section{The inverse function theorem: Version II}
\label{section:invfunccompletefields}

In Theorem~\ref{thm:glinvfunc} we proved an inverse function theorem
that is uniform over a possibly infinite collection of absolute
values.  In this section we work over a single complete field and use
higher dimensional version of Newton's iterative method to prove the
following stronger statement.

\begin{theorem}[Inverse function theorem II]
\label{thm:strong-inverse-function-all-absolute-values}
Let~$\bigl(K,|\,\cdot\,|\bigr)$ be a complete field.  Let~$W, V$ be
smooth quasi-projective varieties over~$K$, and let
$\f\colon{W}\longrightarrow{V}$ be a generically finite generically
\'etale morphism.  Let~$E \subset W$ be the closed subscheme defined
by the~$0$-th fitting ideal sheaf of~$\Omega_{W/V}$.
Fix arithmetic distance functions~$\d_{W}, \d_{V}$, a local height
function~$\l_{E}$, and a boundary function~$\l_{ \partial V}$.
Let~$B \subset W(K)$ be a bounded
subset.

Then there are constants
$\Cl{IFT1a},\Cl{IFT2a},\Cl{IFT3a},\Cl{IFT4a}>0$ and a bounded
subset~$\widetilde{B}\subset{W}(K)$ containing~$B$ such that for
all~$P\in{B}$ and~$q\in{V}(K)$ satisfying
\[
 P \notin E \quad\text{and}\quad
 \d_{V}(\f(P),q) \geq 2 \l_{E}(P) + \Cr{IFT1a} \l_{ \partial V}(q) + \Cr{IFT2a},
\]
there is a unique~$Q\in\widetilde{B}$ satisfying
\[
 \f(Q) = q \quad\text{and}\quad
 \d_{W}(P, Q) \geq \d_{V}(\f(P), q) - \l_{E}(P) - \Cr{IFT3a} \l_{ \partial V}(q) - \Cr{IFT4a}.
 \]
\end{theorem}

\begin{remark}
If~$W$ is projective, then we may take~$B=\widetilde{B}={W}(K)$, since
projective varieties are covered by finitely many affine bounded
subsets.
\end{remark}

\begin{remark}
If~$\f$ is a finite morphism, then~$E$ is equal to the ramification
divisor of~$\f$.
\end{remark}

\begin{corollary}[A variant of continuity of roots]\label{cor:stcontrootsnonarch}
Let~$\bigl(K,|\,\cdot\,|\bigr)$ be a complete field.
Let~$D\in\RR_{>0}$ and~$n\in\ZZ_{>0}$.
Then there are positive constants~$\Cl{IFTx1}, \Cl{IFTx2}>0$ such that the following holds.
Suppose that\textup:
\begin{parts}
  \Part{\textbullet}
  $f, g\in K[t]$ are monic polynomials of degree~$n$;
  \Part{\textbullet}
  $|f| \leq D$ and~$|g| \leq D$;
  \Part{\textbullet}
  There is an ~$\a\in K$  such that
  \[
  f(\a)=0\quad \text{and}\quad |f-g| \leq \Cr{IFTx1} |f'(\a)|^{2}.
  \]
\end{parts}
Then there is~$\b \in K$ such that
\[
g(\b)=0 \quad\text{and}\quad |\a - \b||f'(\a)| \leq \Cr{IFTx2} |f-g|.
\]
\end{corollary}
\begin{proof}
The proof is the same as the proof of Proposition
\ref{prop:controots}, except that we use Theorem
\ref{thm:strong-inverse-function-all-absolute-values}, instead of
Theorem~\ref{thm:glinvfunc}.  Note that if
$\f\colon{W}\longrightarrow{V}$ is a finite surjective morphism of
smooth varieties, then the closed subscheme defined by the~$0$-th
Fitting ideal of~$\Omega_{W/V}$ is equal to the ramification divisor
of~$\f$.
\end{proof}

We first prove the uniqueness of the~$Q$ in
Theorem~\ref{thm:strong-inverse-function-all-absolute-values}.  Then,
in order to prove existence, we consider the archimedean and
non-archimedean cases separately.  
We start with a lemma that says
if~$x$ is in a bounded subset, then all~$y$ that are sufficiently
close to~$x$ also lie in a bounded subset.

\begin{lemma}
\label{lem:closetobdd}
Let~$\bigl(K,|\,\cdot\,|\bigr)$ be a complete field.  Let~$X$ be a
quasi-projective variety over~$K$.  Fix an arithmetic distance
function~$\d_{X}$ and a boundary function~$\l_{\partial(X\times{X})}$.
Let~$B\subset{X}(K)$ be a bounded subset.
\begin{parts}
\Part{(a)}
There are constants~$\Cl{IFT1},\Cl{IFT2}>0$ so that
\[
\bigcup_{x\in B} \bigl\{ y\in X(K) : 
\d_{X}(x,y) \geq \Cr{IFT1}\l_{ \partial(X \times X)}(x,y) + \Cr{IFT2} \bigr\}
\]
is confined in a bounded subset~$B'\subset{X}(K)$.
\Part{(b)}
If~$B$ is an affine bounded subset of some open affine subset
$U\subset{X}$, then~$B'$ may be chosen to be an affine bounded subset
of the same open set~$U$.
\Part{(c)}
If further the absolute value is non-archimedean and~$B$ is a standard
bounded subset of~$U$, then we may take~$B'=B$.
\end{parts}
\end{lemma}
\begin{proof}
Without loss of generality, we may assume that there is an open affine
subset~$U \subset X$ such that $B \subset U(K)$ and that~$B$ is affine
bounded in~$U$.  Let~$x \in B$ and~$y \in X(K)$ be two points that
satisfy
\begin{align*}
\d_{X}(x,y) \geq \Cr{IFT1}\l_{ \partial(X \times X)}(x,y) + \Cr{IFT2}.
\end{align*}

Let~$Z = X \setminus U$. We equip~$Z$ with its reduced structure and
fix a local height~$\l_{Z}$.  The triangle
inequality~\cite[Prop.~3.1(c)]{silverman:arithdistfunctions} gives
\begin{align*}
\min\{ \l_{Z}(y), \d_{X}(x,y) \} \leq \l_{Z}(x) + \Cl{IFT3}\l_{ \partial (X \times X)}(x,y) + \Cl{IFT4},
\end{align*}
where~$\Cr{IFT3}$ and~$\Cr{IFT4}$ are independent of~$x$ and~$y$.

Suppose that
\begin{equation}
  \label{eqn:dXxylelZxIFT3}
  \d_{X}(x,y) \leq \l_{Z}(x) + \Cr{IFT3}\l_{ \partial (X \times X)}(x,y) + \Cr{IFT4}.
\end{equation}
Then we have
\begin{align*}
\Cr{IFT1}\l_{ \partial(X \times X)}(x,y) + \Cr{IFT2} \leq \l_{Z}(x) + \Cr{IFT3}\l_{ \partial (X \times X)}(x,y) + \Cr{IFT4}.
\end{align*}
Since~$x \in B$, we know that
\[
\l_{Z}(x) \leq \sup_{\xi \in B} \l_{Z}(\xi) =: M < \infty,
\]
and hence
\begin{equation}
  \label{eqn:DIFT1CIFT3lXX}
(\Cr{IFT1}-\Cr{IFT3})\l_{ \partial(X \times X)}(x,y) + \Cr{IFT2} - \Cr{IFT4} \leq M.
\end{equation}
Since~$\l_{ \partial(X \times X)}$ is bounded below, we see
that~\eqref{eqn:DIFT1CIFT3lXX} is false for sufficiently
large~$\Cr{IFT1}$ and~$\Cr{IFT2}$ and hence~\eqref{eqn:dXxylelZxIFT3}
is also false.
 
We may thus assume that
\begin{align*}
\l_{Z}(y) \leq \l_{Z}(x) + \Cr{IFT3}\l_{ \partial (X \times X)}(x,y) + \Cr{IFT4}.
\end{align*}
In particular,~$\l_{Z}(y) < \infty$ and~$y \in U$.

Let~$f_{1},\dots,f_{r}$ be~$K$-algebra generators for the ring~$\Ocal(U)$.
Then there are constants~$\Cl{IFT5},\Cl{IFT6}$ such that
\begin{align*}
\d_{X}(\xi, \eta) \leq & \log  \frac{1}{\max_{1 \leq i \leq r}\{ |f_{i}(\xi) - f_{i}(\eta)| \}} \\[2mm]
&+ \Cr{IFT5} (\l_{Z}(\xi) + \l_{Z}(\eta) + \l_{ \partial(X \times X)}(\xi, \eta)) + \Cr{IFT6}
\end{align*}
for all~$\xi,\eta\in{U}(K)$.
Plugging in~$\xi=x$ and~$\eta=y$, we get
\begin{align*}
\Cr{IFT1}  \l_{ \partial(X \times X)} & (x,y) + \Cr{IFT2}  
\leq  \d_{X}(x,y) \\
& \leq \log  \min_{1 \leq i \leq r} \bigl|f_{i}(x) - f_{i}(y)\bigr|^{-1}\\
&\quad{} + \Cr{IFT5} \bigl(\l_{Z}(x) + \l_{Z}(y) + \l_{ \partial(X \times X)}(x, y)\bigr) + \Cr{IFT6}\\
& \leq \log  \min_{1 \leq i \leq r} \bigl|f_{i}(x) - f_{i}(y)\bigr|^{-1}\\
&\quad{} + \Cr{IFT5} \bigl(2\l_{Z}(x)  + (\Cr{IFT3} + 1)\l_{ \partial(X \times X)}(x, y) + \Cr{IFT4}\bigr) + \Cr{IFT6}\\
&\leq  \log  \min_{1 \leq i \leq r} \bigl|f_{i}(x) - f_{i}(y)\bigr|^{-1}\\
&\quad{} + \Cl{IFT7}\l_{Z}(x)  + \Cl{IFT8} \l_{ \partial(X \times X)}(x, y)  + \Cl{IFT9}.
\end{align*}
Thus 
\begin{multline*}
(\Cr{IFT1}- \Cr{IFT8})\l_{ \partial(X \times X)}(x, y)   + \Cr{IFT2} - \Cr{IFT9}- \Cr{IFT7}M \\
 \leq   \log  \min_{1 \leq i \leq r} \bigl|f_{i}(x) - f_{i}(y)\bigr|^{-1}.
\end{multline*}
The boundary function~$\l_{ \partial(X \times X)}$ is bounded below, and the values
$\bigl|f_{1}(x)\bigr|,\ldots,\bigl|f_{r}(x)\bigr|$ are bounded
for~$x\in{B}$, so we see that the values
$\bigl|f_{1}(y)\bigr|,\ldots,\bigl|f_{r}(y)\bigr|$ are bounded.
Hence~$y$ is contained in an affine bounded subset of~$U(K)$.

Finally let~$\bigl(K,|\,\cdot\,|\bigr)$ be non-archimedean, and
let~$B$ be a standard bounded subset of~$U$.  We choose~$K$-algebra
generators~$f_{1},\dots,f_{r}$ and positive constants~$b_1,\ldots,b_r$
so that
\[
B =\bigl \{ x \in U(K) : \bigl|f_{1}(x)\bigr| \leq b_{1}, \dots , \bigl|f_{r}(x)\bigr| \leq b_{r}\bigr\}.
\]
Choosing~$\Cr{IFT1}$ and~$\Cr{IFT2}$ sufficiently large
ensures that
\text{$\bigl|f_{i}(x)-f_{i}(y)\bigr|\leq{b}_{i}$} for all~$i$.
This implies~$\bigl|f_{i}(y)\bigr|\leq{b}_{i}$,
since~$\bigl(K,|\,\cdot\,|\bigr)$ is non-archimedean, and thus we see
that~$y\in{B}$.
\end{proof}

\begin{proof}[Proof of uniqueness in Theorem $\ref{thm:strong-inverse-function-all-absolute-values}$]
As usual, we may assume that $\l_{ \partial{V}}\geq0$. Suppose now
that both of the points~$Q,Q'\in\widetilde{ B}$ satisfy the conclusion
of Theorem~\ref{thm:strong-inverse-function-all-absolute-values}. Our
goal is to show that~$Q=Q'$.

Since~$\f(P)$ varies over the bounded subset~$\f(B)$ and~$\l_{E}$ is
bounded below, Lemma~\ref{lem:closetobdd} tells us that there is a
bounded subset~$B'\subset{V}(K)$ such that $q\in{B'}$ if we
take~$\Cr{IFT1a},\ldots,\Cr{IFT4a}$ to be sufficiently large.
Set
\[
M = \sup_{\eta \in B'} \l_{ \partial V}(\eta) < \infty.
\]
Note that if we increase the values of~$\Cr{IFT1a},\ldots,\Cr{IFT4a}$,
we can use the same   bounded set~$B'$ and the same value for~$M$.
We estimate
\begin{align*}
\min\{ \d_{W}(P, Q), \d_{W}(P, Q') \} &\geq \d_{V}(\f(P), q)-\l_{E}(P) - \Cr{IFT3a} \l_{ \partial V}(q) - \Cr{IFT4a}\\
& \geq \l_{E}(P) + \Cr{IFT2a} - \Cr{IFT4a} - \Cr{IFT3a} M.
\end{align*}
On the other hand, by the triangle
inequality~\cite[Prop.~3.1(c)]{silverman:arithdistfunctions} and
separation (Proposition~\ref{prop:separation}), we have
\begin{align*}
\min\{ \d_{W}(P, Q), \d_{W}(P, Q') \} \leq \d_{W}(Q, Q') +  \Cl{IFT5a}
 \leq \l_{E}(Q) + \Cl{IFT6a}
\end{align*}
for constants~$\Cr{IFT5a},\Cr{IFT6a}$ independent of~$P,q,Q,Q'$.

By the triangle inequality again, we have
\begin{align*}
\min\{ \l_{E}(Q), \d_{W}(P,Q)  \} \leq \l_{E}(P) + \Cl{IFT7a}
\end{align*}
for some constant~$\Cr{IFT7a}$. 
If~$\d_{W}(P,Q) \leq \l_{E}(P) + \Cr{IFT7a}$, then
\[
\l_{E}(P) + \Cr{IFT2a} - \Cr{IFT4a} - \Cr{IFT3a} M \leq \l_{E}(P) + \Cr{IFT7a},
\]
which gives a contradiction if~$\Cr{IFT2a}$ is sufficiently large.

On the other hand, if~$\l_{E}(Q)\leq\l_{E}(P)+\Cr{IFT7a}$, then
\begin{align*}
\l_{E}(P) + \Cr{IFT2a} - \Cr{IFT4a} - \Cr{IFT3a} M \leq \l_{E}(Q) + \Cr{IFT6a} \leq  \l_{E}(P) + \Cr{IFT6a} + \Cr{IFT7a},
\end{align*}
which again  gives a contradiction if~$\Cr{IFT2a}$ is sufficiently large.
\end{proof}

\begin{proof}[Proof of existence of Theorem $\ref{thm:strong-inverse-function-all-absolute-values}$]
The bulk of the proof of the existence of~$Q$ is based on local
calculations that are somewhat different in the archimedean and
non-archimedean cases.  So in the remainder of this section we give
the part of the proof that is common to both cases, and refer the
reader to Sections~\ref{section:newtonnonarch}
and~\ref{section:newtonarch} for the remainder of the proof.

As usual, we may assume that take~$\l_{\partial{V}}\geq0$.  We start
with an open affine cover~$\{ V_{i}\}_{i=1}^{r}$ of~$V$ such that
each~$V_{i}$ admits an \'etale morphism to an affine
space~$\AA^{N}_{K}$.  For each~$i$ we fix an open affine
cover~$\{U_{ij}\}_{j=1}^{s_{i}}$ of~$\f^{-1}(V_{i})$ such that
each~$U_{ij}$ admits an \'etale morphism to~$\AA^{N}$.

We choose generators for the rings~$\Ocal_V(U_{ij})$
so that the associated standard bounded subsets~$B_{ij}\subset{U}_{ij}(K)$
satisfy~$B\subset\bigcup_{i,j}B_{ij}$.
By Lemma~\ref{lem:closetobdd}, there are constants~$\Cl{IFT1b}, \Cl{IFT2b}$, and standard bounded subsets 
$B_{i}' \subset V_{i}(K)$ such that:
\begin{parts}
\Part{\textbullet}
  $\f(B_{ij}) \subset B_{i}'$
\Part{\textbullet}
  if~$q \in V(K)$ satisfies
\[
\d_{V}(\f(P), q) \geq 2 \l_{E}(P) + \Cr{IFT1b} \l_{ \partial V} (q) + \Cr{IFT2b}
\]
for some~$P \in B_{ij}$, then~$q \in B_{i}'$.
\end{parts}

Thus we can reduce the existence of~$Q$ to the situation summarized in
Figure~\ref{figure:PtoBprime}, which essentially reduces the problem
to affine space.

\begin{figure}[ht]
 \[
\xymatrix{
  P \in B_{ij}  \ar@{}[r]|{\subset} \ar@/^{25pt}/[rrr]
  & U_{ij}(K) \ar[r]^{\f} \ar[d]
  & V_{i}(K) \ar[d] 	& B'_{i} \ni q \ar@{}[l]|{\supset}\\
 & \AA^{N}_{K}(K) 	& \AA^{N}_{K}(K). 		& 
}
\]
\caption{Maps from affine bounded sets to affine space}
\label{figure:PtoBprime}
\end{figure}
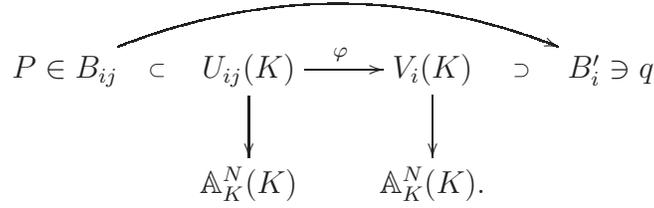

\paragraph{\textbf{Case 1: $\boldsymbol{\bigl(K,|\,\cdot\,|\bigr)}$ is non-archimedean}}
In this case, the proof of existence will be done in
Section~\ref{section:newtonnonarch}; see
Proposition~\ref{prop:newton-nonarhi}.

\paragraph{\textbf{Case 2: $\boldsymbol{\bigl(K,|\,\cdot\,|\bigr)}$ is archimedean}}
We may take~$K=\RR$ or~$\CC$ with the usual absolute value.
Then for any~$x\in{B}_{ij}$, there are compact neighborhoods
\[
x \in T_{1} \subset T_{2} \subset U_{ij}(K)
\quad\text{and}\quad
\f(x) \in T' \subset U_{i}(K)
\]
such that: 
\begin{parts}
\Part{\textbullet}
  $T_{1} \subset T_{2}^{\circ}$, where~$T_{2}^{\circ}$ denotes the interior of~$T_{2}$;
\Part{\textbullet}
  $\f(T_{1}) \subset (T')^{\circ}$;
\Part{\textbullet}
  a small open neighborhood of~$T_{2}$ is homeomorphic to its image by~$U_{ij} \longrightarrow \AA^{N}_{K}$;
\Part{\textbullet}
  a small open neighborhood of~$T'$ is homeomorphic to its image by~$V_{i} \longrightarrow \AA^{N}_{K}$.
\end{parts}
Since~$B_{ij}$ is compact, we can cover~$B_{ij}$ by finitely many sets that look like~$T_{1}$.
Further, if~$\Cr{IFT1b}, \Cr{IFT2b}$ are sufficiently large, 
then for all~$q\in{B}_{ij}$ such that  there exists  some~$P\in{T}_{1}$  with
\begin{align*}
\d_{V}(\f(P), q) \geq 2 \l_{E}(P) + \Cr{IFT1b} \l_{ \partial V}(q) + \Cr{IFT2b},
\end{align*}
we necessarily have~$q\in{T'}$.

In this case, Proposition~\ref{prop:newtonmethod-arch}, which we prove in
Section~\ref{section:newtonarch}, completes the proof of existence,
once we observe that
(1)~the function~$e^{-\d_{W}}$ is  comparable to the usual metric on~$T_{2}$;
(2)~the function~$e^{-\d_{V}}$ is  comparable to the usual metric  on~$T'$;
(3)~if~$J$ is a local equation for~$E$, then~$e^{-\l_{E}}$ is comparable to~$|J|$;
(4)~the function~$\l_{\partial{V}}$ is bounded on~$T'$.
\end{proof}

\section{Newton's method (Non-archimedean case)}
\label{section:newtonnonarch}

When we work over a non-archimedean complete field such as~$\CC_p$, we
lose the local compactness, and thus we cannot localize problems as is
done over the complex numbers.  Nevertheless, similar arguments work
if we use bounded subsets instead of small neighborhoods.  The goal of
this section is to prove the following proposition.

\begin{proposition}\label{prop:newton-nonarhi}
Let~$\bigl(K,|\,\cdot\,|\bigr)$ be a non-archimedean complete field.
Let~$V/K$ and~$W/K$ be  smooth affine varieties, 
let~$\f \colon W \longrightarrow V$ be a generically finite generically \'etale morphism defined over~$K$
and suppose that there are \'etale morphisms~$\pi$ and~$\nu$ to affine space
as in the following diagram\textup:
 \[
\xymatrix{
W \ar[r]^{\f} \ar[d]_{\pi} & V \ar[d]^{\nu}\\
\AA^{N}_{K} & \AA^{N}_{K}.
}
\] 
Let~$E \subset W$ be the closed subscheme defined by the~$0$th Fitting
ideal sheaf~$\Fit_{0}(\Omega_{W/V})$ of the relative sheaf of
differentials~$\Omega_{W/V}$.  Fix arithmetic distance
functions~$\d_{W}$ and~$\d_{V}$ and a local height function~$\l_{E}$.

Let~$B\subset W(K)$ and~$B' \subset V(K)$ be standard bounded subsets such that~$\f(B)\subset B'$.
Then there are constants~$\Cl{nmn1}, \Cl{nmn2}>0$ such that the following holds.
If~$P\in{B}$ and~$q\in{B'}$ are points satisfying
\[
P \notin E\quad\text{and}\quad
\d_{V}(\f(P),q) \geq 2\l_{E}(P) + \Cr{nmn1},
\]
then there exists a point~$Q\in{B}$ satisfying
\[
\f(Q)=q
\quad\text{and}\quad
\d_{W}(P,Q) \geq \d_{V}(\f(P),q) - \l_{E}(P) - \Cr{nmn2}.
\]
\end{proposition}

We start with some preliminary results.  The following proposition,
which we will apply to \'etale morphism from an affine open subset to
an affine space, enables us to go from affine coordinates back to our
original variety.  It is exactly the inverse function theorem for
\'etale morphisms.

\begin{proposition}\label{prop:invet}
Let~$\bigl(K,|\,\cdot\,|\bigr)$ be a non-archimedean complete field.
Let~$X/K$ and~$Y/K$ be affine varieties, and
let~$\pi\colon{X}\longrightarrow{Y}$ be an \'etale morphism defined
over~$K$.  Let~$B\subset{X}(K)$ be a standard bounded set, and
let~$B'\subset{Y}(K)$ be a bounded set.  Fix arithmetic distance
functions~$\d_{X}$ and~$\d_{Y}$ on~$X$ and~$Y$.  Then there are
constants~$\Cl{invet1},\Cl{invet2}>0$ such that if~$x\in{B}$
and~$y\in{B'}$ satisfying
\[
\d_{Y}(\pi(x), y) \geq \Cr{invet1},
\]
then there exists a unique~$z\in{B}$ satisfying
\[
\pi(z)=y
\quad\text{and}\quad
\d_{X}(x,z)  \geq \d_{Y}(\pi(x),y) -\Cr{invet2}.
\]
\end{proposition}
\begin{proof}
\textbf{(Proof of Existence):}
\textbf{Step 1.}\enspace
We first localize the problem so that we can use Chevalley's structure theorem of \'etale morphisms.
Let
\[
X = \Spec S
\quad\text{and}\quad
Y = \Spec R.
\]
Let elements~$f_{1},\dots ,f_{r}\in{S}$,  $g_{1},\dots,g_{r},g_{r+1},\dots,g_{s}\in{R}$
be any collections of elements that have the following properties: 
\begin{itemize}
\item $(f_{1},\dots,f_{r})=S$.
\item $(g_{1},\dots ,g_{s}) = R$.
\item $g_{1},\dots , g_{s}$ generate~$R$ as a~$K$-algebra.
\item For each $i=1, \dots , r$, we have $\pi(\Spec S_{f_{i}}) \subset \Spec R_{g_{i}}$, i.e., 
 \[
\xymatrix{
X = \Spec S \ar[r]^{\pi}   & \Spec R = Y\\
\Spec S_{f_{i}} \ar[r] \ar@{^{(}->}[u] & \Spec R_{g_{i}} \ar@{^{(}->}[u].
}
\]
\end{itemize}

For each~$i=1,\ldots,r$ we define a subset~$B_i\subset{B}$ by
\[
B_{i} = \left\{ x \in B  : \  \left|\frac{f_{1}}{f_{i}}(x)\right| \leq 1, \dots, \left|\frac{f_{r}}{f_{i}}(x)\right| \leq 1   \right\}.
\]
Then~$B_{i}$ is a standard bounded subset of~$\Spec S_{f_{i}}$
and~$\bigcup_{i=1}^{r} B_{i}=B$.  Note that for each~$i=1,\ldots,r$,
the set~$\pi(B_{i})$ is a bounded subset of~$\Spec R_{g_{i}}$.

We claim that there is a number~$\Cl{j0}>0$ such that 
\[
B'_{i} := \{ y \in B'  \mid \d_{Y}(\pi(x), y) \geq \Cr{j0} \text{ for some~$x \in B_{i}$}\}
\]
is a bounded subset of~$\Spec R_{g_{i}}$ for~$i=1,\dots, r$.
Indeed, we may assume that~$\d_{Y}$ is 
\begin{align*}
\d_{Y}(y,y') = \log \frac{1}{ \max_{1\leq j \leq s}\{ |g_{j}(y)-g_{j}(y') |\}}
\end{align*}
on~$(\pi(B)\cup B') \times (\pi(B)\cup B')~$.  Since~$1/g_{i}$ is
regular on~$\Spec R_{g_{i}}$ and~$\pi(B_{i})$ is a bounded subset
of~$\Spec R_{g_{i}}$, there is~$\Cl{j1}>0$ such that
\[
|g_{i}(\pi(x))| \geq \Cr{j1} \quad \text{for all~$x \in B_{i}$}.
\]
Thus, for any~$y \in B'$, if~$\Cr{j0}>0$ is large enough
and~$\d_{Y}(\pi(x), y) \geq \Cr{j0}$ for some~$x \in B_{i}$, then we
get~$|g_{i}(y)| \geq \Cr{j1}$. This proves the claim.

Hence to prove existence, we may
replace~$\pi\colon\Spec{S}\longrightarrow\Spec{R}$,~$B$, and~$B'$ with
$\pi\colon\Spec{S}_{f_{i}}\longrightarrow\Spec{R}_{g_{i}}$,~$B_{i}$,
and~$B'_{i}$.

\textbf{Step 2.}\enspace
By Step 1, we may assume that $Y=\Spec{R}$ and~$X=\Spec{S}$, where
\[
S = \Bigl(R[t]/\bigl(f(t)\bigr) \Bigr)_{g(t)}.
\]
Here~$f(t),g(t)\in{R}[t]$ are polynomials with~$f(t)$ a monic, and
\[
f'(t) \in \Bigl(R[t]/\bigl(f(t)\bigr) \Bigr)^{\times}_{g(t)}.
\]
The situation is summarized in the following diagram:
 \[
\xymatrix@C=90pt{
  \smash[b]{ \Spec\Bigl(R[t]/\bigl(f(t)\bigr) \Bigr)_{g(t)} }
  \ar[d]_{\pi} & *+[l]{B \ni x}  \ar@{|->}[d]_{\pi} \\
  \Spec R & *+[l]{B' \ni y
  \xleftrightarrow[\text{close to one another}]{\text{points that are}} \pi(x)}
}
\]
We write 
\begin{align*}
f(t) &= t^{n} + a_{n-1}t^{n-1} + \cdots + a_{0},\\
g(t) &= b_{m}t^{m} + b_{m-1}t^{m-1} +\cdots + b_{0},
\end{align*}
where~$a_{i},b_{j}\in{R}$.  We write~$f\bigl((a_{i});t\bigr)$
and~$g\bigl((b_{j});t\bigr)$ if we need to specify the coefficients.

We choose~$c_{0},\dots, c_{l} \in R$ so that
\[
R = K[a_{0},\dots, a_{n-1}, b_{0},\dots , b_{m}, c_{0},\dots ,c_{l}].
\]
Then we may take 
\begin{multline*}
  \d_{Y}(y,y')\\
  =\log \min_{i,j,k}\Bigl\{\bigl|a_{i}(y)-a_{i}(y')\bigr|^{-1}, \bigl|b_{j}(y)-b_{j}(y')\bigr|^{-1}, \bigl|c_{k}(y)-c_{k}(y')\bigr|^{-1}\Bigr \}
\end{multline*}
on~$\bigl(\pi(B)\cup{B'}\bigr)\times\bigl(\pi(B)\cup{B'}\bigr)$.

Let~$x \in B$ and~$y \in B'$.  We are going to apply
Proposition~\ref{prop:controots} to the
polynomials~$f\bigl((a_{i}(x));t\bigr)$
and~$f\bigl((a_{i}(y));t\bigr)$, and to the quantity~$t(x)$, which is
a root of~$f\bigl((a_{i}(x));t\bigr)$.
If~$\d_{Y}\bigl(\pi(x),y\bigr)\geq{\Cl{1c}}$, then for
all~$i=0,\dots,n-1$ we get
\begin{equation}
  \label{coeffdist}
  \bigl|a_{i}(\pi(x))- a_{i}(y)\bigr| \leq e^{-\Cr{1c}}.
\end{equation}

Since~$\pi(x)$ and~$y$ move in bounded subsets, there is a number~$D>0$, which is independent of~$x,y$, such that
\[
|a_{i}(\pi(x))| \leq D \quad\text{and}\quad \bigl|a_{i}(y)\bigr| \leq D
\quad\text{for all~$i=0,\dots, n-1$.}
\]
Since~$f'(t)$ and~$1/f'(t)$ are regular functions on~$X$, there is a
$D'>0$, which is independent of~$x$, such that
\[
\frac{1}{D'} \leq \bigl|f'\bigl((a_{i}(x)); t(x)\bigr)\bigr| \leq D'.
\]
Then by Proposition~\ref{prop:controots}, 
if~$\Cr{1c}$ is large enough, there is a number~$\Cl{2c}>0$ such that the following holds.
For all~$x \in B$ and~$y \in B'$ satisfying
\[
\d_{Y}\bigl(\pi(x),y\bigr)\geq \Cr{1c},
\]
there is~$\beta\in{K}$ such that
\begin{equation}
\label{rootbound}
f\bigl((a_{i}(y)\bigr); \beta) = 0
\;\text{and}\;
\bigl|t(x)- \beta\bigr| \leq \Cr{2c} \max_{\leq i \leq n-1}\bigl\{ \bigl|a_{i}(\pi(x)) - a_{i}(y)\bigr|\bigr\}.
\end{equation}
Note that~$(y,\beta)$ defines a~$K$-valued point~$z$ of
$\Spec{R}[t]/\bigl(f(t)\bigr)$.  By~\eqref{rootbound}, the point~$z$
moves in a bounded subset of~$\Spec R[t]/\bigl(f(t)\bigr)$ as~$x$
and~$y$ move.

By the choice of~$\d_{Y}$ and~\eqref{coeffdist} and~\eqref{rootbound},
there is a constant~$\Cl{3c}>0$ independent of~$x$ and~$y$ such that
\begin{align}
\Bigl|g\bigl((b_{j}(\pi(x)));t(x)\bigr) - g( (b_{j}\bigl(\pi(z)));t(z)\bigr)\Bigr| \leq \Cr{3c} e^{-\d_{Y}(\pi(x),y)}  \leq \Cr{3c} e^{-\Cr{1c}}.
\end{align}
Since~$\bigl|g\bigl((b_{j}(\pi(x)));t(x)\bigr)\bigr|^{-1}$ is bounded above when~$x$ runs over~$B$, we see that
\[
\Bigl|g\bigl( (b_{j}(\pi(z)));t(z) \bigr)\Bigr|^{-1}
\]
is also bounded above as~$x$ and~$y$ move.
Hence~$z$ is contained in a bounded subset~$B_{1}$ of~$X = \Spec (R[t]/(f(t)))_{g(t)}$.
On~$(B\cup B_{1}) \times (B\cup B_{1})$, we may take
\[
\d_{X}(\xi, \eta) = \log \min\Bigl\{
e^{\d_{Y}(\pi(\xi),\pi(\eta))},
\bigl|t(\xi)-t(\eta)\bigr|^{-1},
\bigl|(1/g)(\xi) - (1/g)(\eta)  \bigr|^{-1}
\Bigr\}.
\]
Then we get
\begin{align*}
\d_{X}(x,z) &\geq   \log \min\{ e^{\d_{Y}(\pi(x),y)}, \Cr{2c}^{-1}e^{\d_{Y}(\pi(x),y)} , \Cl{bd1/g} e^{\d_{Y}(\pi(x),y)} \}  \\
& \geq \d_{Y}\bigl(\pi(x),y\bigr) - \Cl{4c}\\
&\geq   \Cr{1c} - \Cr{4c}
\end{align*}
where~$\Cr{4c}= \log\max\{1,\Cr{2c}^{-1},\Cr{bd1/g}\}$.
Thus if~$\Cr{1c}$ is large enough, by Lemma~\ref{lem:closetobdd}, we get~$z \in B$.
(Note that we can take~$\Cr{2c}, \Cr{bd1/g}$ independent of~$\Cr{1c}$).
This~$z$ is the point that we want.

\par\noindent\textbf{(Proof of Uniqueness):}\enspace
We apply Proposition~\ref{prop:separation} to the morphism $\pi\colon{X}\longrightarrow{Y}$.
Since~$\Omega_{X/Y} = 0$, there is a  number~$\Cl{3cc}>0$ such that
\[
\d_{X}(z,z') \leq \Cr{3cc}
\]
for all~$z,z' \in B$ such that~$\pi(z)=\pi(z')$ and~$z\neq{z'}$.
Suppose that there are constants~$\Cr{invet1},\Cr{invet2}>0$ and
points~$x,z,z'\in{B}$ and~$y\in{B'}$ satisfying
\begin{parts}
  \Part{\textbullet} 
  $\d_{Y}(\pi(x), y) \geq \Cr{invet1}$;
  \Part{\textbullet}
  $z \neq z'$ and $\pi(z) = \pi(z') = y$;
  \Part{\textbullet}
  $\d_{X}(x,z) + \Cr{invet2} \geq \d_{Y}(\pi(x), y)$ and~$\d_{X}(x,z') + \Cr{invet2} \geq \d_{Y}(\pi(x), y)$.
\end{parts}
Then find that
\begin{align*}
\Cr{invet1} - \Cr{invet2} \leq \d_{Y}(\pi(x), y) - \Cr{invet2} &\leq \min\{\d_{X}(x,z) , \d_{X}(x,z')   \}\\
&\leq \d_{X}(z, z') + \Cl{4cc} \\
& \leq \Cr{3cc} + \Cr{4cc}
\end{align*}
where~$\Cr{4cc}$ comes from the triangle inequality
on~\text{$B\times{B}\times{B}$}.  Thus if~$\Cr{invet1}-\Cr{invet2}$ is
large enough, more precisely if
$\Cr{invet1}>\Cr{invet2}+\Cr{3cc}+\Cr{4cc}$, then we get a
contradiction, which proves that~$z$ is unique.
\end{proof}

We next compare the distance between two points to the distance
between their images under an \'etale morphism.

\begin{corollary}
\label{cor:comparedistet}
Let~$\bigl(K,|\,\cdot\,|\bigr)$ be a non-archimedean complete field.
Let~$X/K$ and~$Y/K$ be affine varieties, and
let~$\pi\colon{X}\longrightarrow{Y}$ be an \'etale morphism defined
over~$K$.  Fix arithmetic distance functions~$\d_{X}$ and~$\d_{Y}$
on~$X$ and~$Y$, respectively.  Let~$B\subset{X}(K)$ be a bounded
subset.  Then there are constants~$\Cl{1d},\Cl{2d}>0$ such that
for all points~$x,y\in B$ satisfying
\[
\d_{X}(x,y) \geq \Cr{1d},
\]
we have
\[
\d_{X}(x,y) +\Cr{2d} \geq  \d_{Y}\bigl(\pi(x),\pi(y)\bigr).
\]
\end{corollary}
\begin{proof}
We may assume~$B$ is a standard bounded subset.  We apply
Proposition~\ref{prop:invet} to the map~$\pi\colon{X}\to{Y}$ and the
bounded sets~$B$ and~$B'=\pi(B)$, and we let~$\Cr{1d}, \Cr{2d}$ be the
constants appearing in the conclusion of that proposition.

In general, there is a constant~$\Cl{3d}$ such that
\[
\d_{X} \leq \d_{Y} \circ (\pi \times \pi) + \Cr{3d}.
\]
Hence if~$\d_{X}(x,y)\geq\Cl{4d}$ with~$x,y\in{B}$, then we have
\[
\d_{Y}(\pi(x), \pi(y)) \geq \Cr{4d} - \Cr{3d}.
\]
Thus if we take~$\Cr{4d}$ so that~$\Cr{4d}-\Cr{3d} \geq  \Cr{1d}$, then
there is a point~$z\in{B}$ such that
\[
\pi(z) = \pi(y) \quad\text{and}\quad
 \d_{X}(x,z) \geq \d_{Y}(\pi(x), \pi(y)) - \Cr{2d}.
\]
Hence
\[
\Cr{4d} - \Cr{3d} - \Cr{2d} \leq \min\{ \d_{X}(x,y), \d_{X}(x,z) \} \leq \d_{X}(y,z) + \Cl{5d},
\]
where~$\Cr{5d}$ comes from the triangle inequality on~$B\times{B}\times{B}$.

We apply Proposition~\ref{prop:separation} to the
morphism~$\pi\colon{X}\to{Y}$.  Since~$\Omega_{X/Y}=0$, there is a
number~$\Cl{6d}>0$ such that
\[
\d_{X}(w,w') \leq \Cr{6d}
\quad\text{for all $w,w'\in B$ with $\pi(w)=\pi(w')$ and~$w \neq w'$.}
\]
In our situation, if~$y\neq{z}$, then~$\d_{X}(y,z)\leq\Cr{6d}$, so we
find that
\[
\Cr{4d} - \Cr{3d} - \Cr{2d} \leq \Cr{6d} + \Cr{5d}.
\]
Thus if we take~$\Cr{4d}$ sufficiently large, i.e., so that~$\Cr{4d}-\Cr{3d}-\Cr{2d}>\Cr{6d}+\Cr{5d}$, then 
$y=z$ and 
\[
\d_{X}(x,z) \geq \d_{Y}\bigl(\pi(x), \pi(y)\bigr)-\Cr{2d},
\]
which concludes the proof of Corollary~\ref{cor:comparedistet}.
\end{proof}

Our next result is an algebraic version of Taylor's theorem up to
second order terms that is uniform on bounded sets.

\begin{lemma}
\label{lem-taylor}
Let~$\bigl(K,|\,\cdot\,|\bigr)$ be a complete field, let~$X/K$ be an
affine variety, fix an arithmetic distance function~$\d_{X}$ on~$X$,
and let~$B\subset{X}(K)$ be a bounded subset of~$X$.
\begin{parts}
\Part{(1)}
Let~$f \in \Ocal(X)$ be a regular function.
Then there is a number~$C>0$ such that for all~$a, b \in B$, we have
\begin{equation}
  \label{eqn:fafbleedxab}
  \bigl|f(a)-f(b)\bigr| \leq C e^{-\d_{X}(a,b)}.
\end{equation}
\Part{(2)}
Let
\[
\pi \colon X \longrightarrow \AA^{N}_{K}=\Spec K[x_{1},\dots,x_{N}]
\]
be an \'etale morphism.
For any regular function~$f\in\Ocal(X)$, define~$\partial{f}/\partial{x}_{i}$ by
\[
\Omega_{X} \simeq \pi^{*}\Omega_{\AA^{N}} \simeq \bigoplus_{i=1}^{N} \Ocal_{X}d(\pi^{*}x_{i}),
\qquad
df \longleftrightarrow \sum_{i=1}^{N} \frac{ \partial f }{ \partial x_{i}}d(\pi^{*}x_{i}).
\]
Then for any~$f\in\Ocal(X)$, there is a number~$C>0$ such that for
all~$a,b\in{B}$, we have
\begin{equation}
  \label{eqn:2ndTaylor}
  \left|f(a)-f(b)-\sum_{i=1}^{N}\frac{ \partial f }{ \partial x_{i}}(b)(a_{i}-b_{i})\right| \leq C e^{- 2\d_{X}(a,b)},
\end{equation}
where~$\pi(a) = (a_{1},\dots , a_{N})$
and~$\pi(b)=(b_{1},\dots,b_{N})$. N.B. The crucial quantity
in~\eqref{eqn:2ndTaylor} is the~$2$ appearing in the exponent
of~$e^{-2\d_X(a,b)}$.
\end{parts}
\end{lemma}
\begin{proof}
(1)\enspace
Let~$f=f_{1},f_{2},\dots,f_{r}$ be~$K$-algebra generators of~$\Ocal(X)$.
Then we may take
\[
\d_{X}(a,b) = \log  \min_{1\leq i \leq r} \bigl|f_{i}(a)-f_{i}(b)\bigr|^{-1}
\]
for~$(a,b)\in{B}\times{B}$. Then~\eqref{eqn:fafbleedxab} is a
tautology.
\par\noindent(2)\enspace

Let~$I$ be the ideal of diagonal in~$X\times{X}$.  By definition we
have~$\Omega_{X}=I/I^{2}$, and the differential of a regular
function~$g\in\Ocal(X)$ is given by~$dg=g\otimes1-1\otimes{g}$.
Thus the image of
\begin{align}
  \label{eq:der}
  F = (f \otimes 1 - 1 \otimes f) - \sum_{i=1}^{N} \left(1 \otimes
  \frac{ \partial f }{ \partial x_{i}} \right) (\pi^{*}x_{i} \otimes 1
  - 1 \otimes \pi^{*}x_{i})
\end{align}
in~$ \Omega_{X} = I/I^{2}$ is 
\begin{align*}
  df -
  \smash[t]{  \sum_{i=1}^{N} \frac{ \partial f }{ \partial x_{i}}d(\pi^{*}x_{i}) = 0. }
\end{align*}
This means that~$F\in{I}^{2}$.
Thus we have
\[
\l_{F} + C\geq 2 \l_{I} = 2 \d_{X} \quad \text{on~$B \times B$}
\]
for some constant~$C>0$.
In other words,
\[
\bigl|F(a,b)\bigr| \leq e^{C} e^{-2\d_{X}(a,b)} \quad \text{for~$(a,b) \in B \times B$}.
\]
Since 
\[
F(a,b) = f(a) - f(b) -
\smash[t]{ \sum_{i=1}^{N}  \frac{ \partial f }{ \partial x_{i}} (b) (a_{i}- b_{i}), }
\]
this gives the desired inequality.
\end{proof}

\begin{proof}[Proof of Proposition $\ref{prop:newton-nonarhi}$]
We fix coordinates of the affine space~$\AA^{N}$.
For~$\xi=(\xi_{1},\dots,\xi_{N})\in{K}^{N}$, we
write~$\|\xi\|=\max_{1\leq{i}\leq{N}}|\xi_{i}|$, and similarly~$\|M\|$
denotes the maximum of the absolute values of the entries of the
matrix~$M$.

We choose numbers~$b,b'>0$ so that the
standard bounded subsets
\begin{align*}
  B_{0}  &= \{ \xi \in K^{N} \mid \| \xi\| \leq b \}\subset \AA^{N}(K), \\
  B'_{0} &= \{ \xi \in K^{N} \mid \| \xi\| \leq b' \} \subset \AA^{N}(K),
\end{align*}
satisfy~$\pi(B)\subset{B}_{0}$ and~$\nu(B')\subset{B'}_{0}$.
The following diagram summarizes our setting:
\par
 \[
\xymatrix{
B \ar[d]_{\pi}  \ar@{}[r]|{\subset} \ar@/^{25pt}/[rrr] & W(K) \ar[r]^{\f} \ar[d]_{\pi} 	& V(K) \ar[d]^{\nu} 	& B' \ar[d]^{\nu} \ar@{}[l]|{\supset}\\
B_{0} 	\ar@{}[r]|{\subset}   & \AA^{N}_{K}(K) 	& \AA^{N}_{K}(K) 		& B'_{0} \ar@{}[l]|{\supset}.
}
\]

Since~$\pi$ and~$\nu$ are \'etale morphisms, we have
 \[
\xymatrix{ 
\f^{*}\Omega_{V}	\ar[r]^{D\f}  &   \Omega_{W} \\
\f^{*} \nu^{*} \Omega_{\AA^{N}_{K}} \ar[u]_{\wr} & \pi^{*}\Omega_{\AA^{N}_{K}} \ar[u]_{\wr}\\
\Ocal_{W}^{N} \ar[u]_{\wr} &  \Ocal_{W}^{N} \ar[u]_{\wr}
}
\]
where the vertical arrows are isomorphisms.
Thus~$D\f$ is represented by an~$N \times N$-matrix with entries in~$\Ocal(W)$.
We identify~$D\f$ with this matrix,
and we let
\[
J = \det D\f.
\]
Since~$\f$ is generically \'etale, we know that~$J$ is a non-zero
regular function on~$W$.  Note that~$J$ generates~$
\Fit_{0}(\Omega_{W/V})$, and hence we may take
\[
\l_{E} = \log {|J|^{-1}}.
\]

We take as our arithmetic distance function~$\d_{\AA^{N}}$
on~$\AA^{N}$ the function
\[
\d_{\AA^{N}}(\xi, \eta) = \log \| \xi - \eta\|^{-1}
\quad\text{for~$\xi, \eta \in K^{N}$,}
\]
or equivalently,
\[
e^{-\d_{\AA^{N}}(\xi, \eta)} = \| \xi - \eta \|.
\]
We henceforth use without comment this identification of the usual norm
on~$K^N$ and the arithmetic distance function.  In particular, we have
\[
\bigl\|\pi(x) - \pi(x')\bigr\|  \ll  e^{-\d_{W}(x, x')}
\quad\text{and}\quad
\bigl\| \nu(y) - \nu(y')\bigr \| \ll e^{-\d_{V}(y, y')}.
\]

We are going to work in the following set:
\begin{align*}
  S = \left\{ (x,y) \in (W \times V)(K)  :
  \begin{array}{@{}l@{}}
    x \in B,\; y \in B',\; J(x) \neq 0\\[.5\jot]
    e^{-\d_{V}(\f(x),y)} \leq \Cl{si-1} \bigl|J(x)\bigr|^{2}
  \end{array}
  \right\},
\end{align*}
where~$\Cr{si-1}$ is a positive number that we will take sufficiently
small during the proof so as to ensure the various desired
properties.  In the following, the labeled constants are
positive numbers that depend only on
\begin{align}
\label{dep} W,V, \f, B, B', B_{0}, B'_{0}, \d_{W}, \d_{V}, \l_{E}, \pi, \nu.
\end{align}
In particular, they do not depend on the points chosen on our
varieties.  Sometimes we omit the phrase such as ``there exists a
constant~$C>0$ such that\dots\thinspace.''

Since~$J$ is a regular function on~$W$, we have
\begin{equation}
  \label{Jbound} \bigl|J(x)\bigr| \leq \Cl{Jbd}
  \quad\text{for all~$x \in B$.}
\end{equation}

\begin{claim}[Pullback via~$\pi$]
\label{claim:pullback}
There are constants~$\Cl{si-atw1},\Cl{si-atw2}>0$ such that
if~$x\in{B}$ and~$\zeta\in{B}_{0}$ satisfy
\[
\|\pi(x) - \zeta \| \leq \Cr{si-atw1},
\] 
then there exists a unique~$z\in{B}$ satisfying
\[
\pi(z) = \zeta
\quad\text{and}\quad
e^{-\d_{W}(x,z)} \leq \Cr{si-atw2}\bigl \| \pi(x) - \zeta\bigr\|.
\]
\end{claim}
We observe that Claim~\ref{claim:pullback} is essentially a
restatement of Proposition~\ref{prop:invet}.

We now state and prove several useful inequalities.
\begin{claim}[Key Inequalities]
\label{claim:keyinequliaties}
There are numbers~$\Cl{si-key1}, \Cl{si-key2}>0$ such that
for all~$(x,y)\in{S}$,  we have\textup:
\begin{align}
  \label{ineq:key1} \bigl\| (D\f)(x)^{-1}\bigl(\nu(\f(x)) - \nu(y)  \bigr) \bigr \|
  &\leq \Cr{si-key1} \frac{1}{|J(x)|} e^{-\d_{V}(\f(x),y)}. \\[3mm]
  \label{ineq:key2} \bigl\| (D\f)(x)^{-1}\bigl(\nu(\f(x)) - \nu(y)  \bigr) \bigr \|
  &\leq \Cr{si-key1} \Cr{si-1}^{1/2} e^{-\d_{V}(\f(x),y)/2}. \\[3mm]
  \label{ineq:key3} \bigl\| (D\f)(x)^{-1}\bigl(\nu(\f(x)) - \nu(y)  \bigr) \bigr \|
  &\leq  \Cr{si-key1} \Cr{si-1} |J(x)| \leq \Cr{si-key2} \Cr{si-1} .
\end{align}
\end{claim}
\begin{proof}[Proof of Claim~$\ref{claim:keyinequliaties}$]
We have
\begin{align*}
  \bigl\| (D\f)(x)^{-1}\bigl(\nu(\f(x)) - \nu(y)  \bigr) \bigr \|
  &\leq \bigl\|(D\f)(x)^{-1}\bigr\|\cdot \bigl\| \nu(\f(x)) - \nu(y)\bigl \|\\
  &\leq \Cr{si-key1} \frac{1}{|J(x)|} e^{-\d_{V}(\f(x),y)}.
\end{align*}
Here the first inequality follows from the triangle inequality.
For the second inequality, write 
\[
(D\f)(x)^{-1} = \frac{1}{J(x)}\adj\bigl((D\f)(x)\bigr)
\]
where~$\adj((D\f)(x))$ is the adjoint matrix.  The entries
of~$\adj((D\f)(x))$ are regular functions on~$W$, and
thus~$\bigl\|\adj((D\f)(x))\bigr\|$ is bounded on~$B$.  This
proves~\eqref{ineq:key1}, and then~\eqref{ineq:key2}
and~\eqref{ineq:key3} follow from~\eqref{ineq:key1}, the inequality
\[
e^{-\d_{V}(\f(x),y)} \leq \Cr{si-1} |J(x)|^{2},
\]
and the fact that~$|J|$ is bounded on~$B$.
\end{proof}

The preceding material allows us to reduce the proof of
Proposition~\ref{prop:newton-nonarhi} to the  following statement:  There are
constants~$\Cr{si-1},\Cl{si-2}>0$ such that for all~$(P,q)\in{S}$, there
exists a point~$Q\in{B}$ satisfying
\[
 \f(Q)=q\quad\text{and}\quad
 e^{-\d_{W}(P,Q)} \leq \Cr{si-2} \frac{1}{\bigl|J(P)\bigr|} e^{-\d_{V}(\f(P),q)}.
\]
So we start with  an arbitrary~$(P,q)\in{S}$, and we will construct the
requisite~$Q$ as the limit of a sequence of
points~$Q_{0},Q_{1},\dots\in{W}(K)$ defined in the following way.
 
We start by choosing~$\Cr{si-1}$ sufficiently small  so that
\[
\Cr{si-key2} \Cr{si-1} \leq \Cr{si-atw1}
\quad\text{and}\quad
\Cr{si-key2} \Cr{si-1}  \leq  b.
\]

\paragraph{\textbf{Algorithm Used to Construct of~$\boldsymbol{Q_{0},Q_{1},\dots}$}}
\begin{parts}
  \Part{(1)}
  Set $Q_{0}=P$.
  \Part{(2)}
  Given~$Q_{0},\dots,Q_{i}$ satisfying~$(Q_{j},q)\in{S}$ for~$j=0,\dots,i$, we consider the quantity
\[
\eta := \pi(Q_{i}) - (D\f)(Q_{i})^{-1}\bigl(\nu\bigl(\f(Q_{i})\bigr) - \nu(q)\bigr) \in \AA^{N}(K) .
\]
Since~$(Q_{i},q)\in{S}$, we know from~\eqref{ineq:key3} that
\begin{align*}
\Bigl\|  (D\f)(Q_{i})^{-1}\bigl(\nu\bigl(\f(Q_{i})\bigr) - \nu(q)\bigr)   \Bigr\| \leq \Cr{si-key2} \Cr{si-1}.
\end{align*}
It follows that
\[
\| \eta \| \leq \max\Bigl\{ \bigl\| \pi(Q_{i})\bigr \| ,  \Cr{si-key2} \Cr{si-1}\Bigr\} \leq b,
\]
so $\eta \in B_{0}$, and also that
\[
  \bigl\| \pi(Q_{i}) - \eta \bigr\| \leq \Cr{si-key2} \Cr{si-1} \leq \Cr{si-atw1}.
\]
Hence  Claim~\ref{claim:pullback} tells us that there is a unique point~$Q_{i+1}\in{B}$ satisfying
\begin{equation}
 \label{ineq:QiQi+1}
 \pi(Q_{i+1}) = \eta\quad\text{and}\quad
 e^{-\d_{W}(Q_{i}, Q_{i+1})} \leq \Cr{si-atw2} \| \pi(Q_{i}) - \eta\|.
\end{equation}
\end{parts}

In order to ensure that we can continue this procedure, and to prove
that~$Q_{i}$ converges to a point~$Q$ having the desired properties,
we verify the following assertions.

\begin{claim}
\label{claim:inductionstep}
Let~$\a$ be an arbitrary real number satisfying $0<\a<1$.
If~$\Cr{si-1}$ is small enough, depending only on~\eqref{dep} and the
choice of~$\a$, then the following are true\textup:
\begin{parts}
  \Part{(1)}
  $(Q_{i+1},q) \in S$, so we can continue the algorithm to create~$Q_{i+2}$.
  \Part{(2)}
 For all~$j = 0, 1, \dots$, we have
 \begin{equation}
  \label{invJ}
  \bigl|J(Q_{j})\bigr| = \bigl|J(P)\bigr|.
\end{equation}
\Part{(3)}
For all~$j = 0, 1, \dots$, we have
\begin{equation}
\label{disfQjq}
e^{-\d_{V}(\f(Q_{j}), q)} \leq \a^{j} e^{-\d_{V}(\f(P),q)}.
\end{equation}
\end{parts}
\end{claim}

We prove Claim~\ref{claim:inductionstep} by induction.  More
precisely, we assume that~\eqref{invJ} and~\eqref{disfQjq} are true
for~$j=0,\dots,i$ and we prove that~$(Q_{i+1},q)\in{S}$ and
that~\eqref{invJ} and~\eqref{disfQjq} are true for~$j=i+1$.  In
this induction step, we may replace~$\Cr{si-1}$ with a smaller value,
but the value of~$\Cr{si-1}$ is always independent of~$i$.

First, we note that
\begin{align*}
&e^{-\d_{V}(\f(Q_{i+1}), q)} \\
& \leq \Cl{si-tri1} \max\{ e^{-\d_{V}(\f(Q_{i+1}), \f(Q_{i}))}, e^{-\d_{V}(\f(Q_{i}), q)} \}   \!  &&   \text{triangle inequality}  \\
& \leq \Cr{si-tri1} \max\{ e^{-\d_{V}(\f(Q_{i+1}), \f(Q_{i}))}, \Cr{si-1} \Cr{Jbd}^{2} \}  &&	\text{by~$(Q_{i},q)\in S$ and~\eqref{Jbound}}	\\
& \leq \Cr{si-tri1} \max\{ \Cl{si-df} e^{-\d_{W}(Q_{i+1},Q_{i})} ,  \Cr{si-1} \Cr{Jbd}^{2} \}  && \text{$\d_{W} \ll \d_{V} \circ (\f \times \f) + O(1)$}	\\
& \leq \Cr{si-tri1} \max\{ \Cr{si-df}\Cr{si-atw2} \| \pi(Q_{i}) - \eta\| ,  \Cr{si-1} \Cr{Jbd}^{2} \}  && \text{by~\eqref{ineq:QiQi+1}}	\\
& \leq \Cr{si-tri1} \max\{ \Cr{si-df}\Cr{si-atw2} \Cr{si-key2} \Cr{si-1},  \Cr{si-1} \Cr{Jbd}^{2} \} && \text{by the construction of~$\eta$}	\\
& = \Cl{1stbdfQi+1q} \Cr{si-1},
\end{align*}
where~$\Cr{1stbdfQi+1q} = \Cr{si-tri1} \max\{ \Cr{si-df}\Cr{si-atw2} \Cr{si-key2} ,   \Cr{Jbd}^{2} \}$.
Therefore, if~$\Cr{si-1}$ is sufficiently small, then we may apply Corollary~\ref{cor:comparedistet} 
to the \'etale morphism~$\nu\colon{V}\to\AA^{N}_{K}$ and the points~$\f(Q_{i+1})$ and~$q$ to obtain
\[
e^{-\d_{V}(\f(Q_{i+1}),q)} \leq \Cl{disqQi+1nu} \bigl\| \nu(\f\bigl(Q_{i+1}\bigr)) - \nu(q) \bigr\|.
\]
Hence 
\begin{align*}
&\hspace{-1em}e^{-\d_{V}(\f(Q_{i+1}),q)} \\
& \leq \Cr{disqQi+1nu} \bigl\| \nu(\f(Q_{i+1})) - \nu(q) \bigr\|\\
& \leq \Cr{disqQi+1nu} \bigl\| \nu(\f(Q_{i+1})) - \nu(\f(Q_{i})) - (D\f)(Q_{i})(\pi(Q_{i+1})- \pi(Q_{i})) \\
&\qquad \qquad  \qquad  +  (D\f)(Q_{i})(\pi(Q_{i+1})- \pi(Q_{i})) + \nu(\f(Q_{i}))    - \nu(q)  \bigr\|\\
& = \Cr{disqQi+1nu} \bigl\| \nu(\f(Q_{i+1})) - \nu(\f(Q_{i})) - (D\f)(Q_{i})(\pi(Q_{i+1})- \pi(Q_{i}))\bigr\|\\
& \leq \Cr{disqQi+1nu} \Cl{si-taylor2} e^{-2 \d_{W}(Q_{i+1},Q_{i})},
\end{align*}
where the third equality follows from the construction of~$Q_{i+1}$
and the last inequality follows from Lemma~\ref{lem-taylor}(2).
By the construction of~$Q_{i+1}$, we get
\begin{align}
 e^{-\d_{V}(\f(Q_{i+1}),q)} 
&\leq \Cr{disqQi+1nu} \Cr{si-taylor2} e^{-2 \d_{W}(Q_{i+1},Q_{i})} \notag\\
& \leq  \Cr{disqQi+1nu} \Cr{si-taylor2}  \Cr{si-atw2}^{2} \bigl\| \pi(Q_{i}) - \eta\bigr\|^{2} \notag\\
 \label{fQi+1qtaylor}
 & = \Cr{disqQi+1nu} \Cr{si-taylor2}  \Cr{si-atw2}^{2} \bigl\|  (D\f)(Q_{i})^{-1}(\nu(\f(Q_{i})) - \nu(q))   \bigr\|^{2}  \\
 \label{fQi+1qJQi2}
 & \leq \Cr{disqQi+1nu} \Cr{si-taylor2}  \Cr{si-atw2}^{2}  \Cr{si-key1}^{2} \Cr{si-1}^{2} |J(Q_{i})|^{2},
\end{align}
where the last inequality follows from the fact~$(Q_{i},q)\in{S}$ and~\eqref{ineq:key3}.

Now note that
\begin{align*}
\bigl|J(Q_{i+1}) - J(Q_{i})\bigr| 
& \leq \Cl{Jdiff} e^{-\d_{W}(Q_{i+1}, Q_{i})}
&&  \text{by Lemma~\ref{lem-taylor}(1)}\\ 
& \leq \Cr{Jdiff} \Cr{si-atw2} \bigl\| \pi(Q_{i}) - \eta\bigr\|
&&  \text{by construction of~$Q_{i+1}$} \\
& = \Cr{Jdiff} \Cr{si-atw2} \bigl\|  (D\f)(Q_{i})^{-1}(\nu(\f(Q_{i})) - \nu(q))   \bigr\| \hidewidth \\
&&& \omit\hfill\text{by definiton of~$\eta$}\\
& \leq  \Cr{Jdiff} \Cr{si-atw2}  \Cr{si-key1} \Cr{si-1} |J(Q_{i})|
&& \text{by~\eqref{ineq:key3}}.
\end{align*}
Thus if~$\Cr{si-1}$ is small enough so that~$\Cr{Jdiff}\Cr{si-atw2}\Cr{si-key1}\Cr{si-1}<1$,
then
\[
|J(Q_{i+1})| = |J(Q_{i})| = |J(P)|.
\]
In particular, this shows that that~$J(Q_{i+1})\neq0$ and
verifies~\eqref{invJ} for~$j=i+1$.

Plugging~$|J(Q_{i+1})| = |J(Q_{i})|$  into~\eqref{fQi+1qJQi2}, we get
\begin{align*}
e^{-\d_{V}(\f(Q_{i+1}),q)} \leq  \Cr{disqQi+1nu} \Cr{si-taylor2}  \Cr{si-atw2}^{2}  \Cr{si-key1}^{2} \Cr{si-1}^{2} |J(Q_{i+1})|^{2}.
\end{align*}
Hence if~$\Cr{si-1}$ is sufficiently small to ensure
that~$\Cr{disqQi+1nu}\Cr{si-taylor2}\Cr{si-atw2}^{2}\Cr{si-key1}^{2}\Cr{si-1}\leq1$,
then we get
\[
e^{-\d_{V}(\f(Q_{i+1}),q)} \leq \Cr{si-1} \bigl|J(Q_{i+1})\bigr|^{2}.
\]
This completes the proof that~$(Q_{i+1},q)\in{S}$.

Finally, we prove~\eqref{disfQjq} for~$j=i+1$.
By~\eqref{fQi+1qtaylor} and~\eqref{ineq:key2}, we have
\begin{align*}
  e^{-\d_{V}(\f(Q_{i+1}), q)}
  & \leq  \Cr{disqQi+1nu} \Cr{si-taylor2}  \Cr{si-atw2}^{2} \Bigl\|  (D\f)(Q_{i})^{-1}\Bigl(\nu\bigl(\f(Q_{i})\bigr) - \nu(q)\Bigr)   \Bigr\|^{2} \\
  & \leq  \Cr{disqQi+1nu} \Cr{si-taylor2}  \Cr{si-atw2}^{2} \Cr{si-key1}^{2} \Cr{si-1} e^{-\d_{V}(\f(Q_{i}),q)}.
\end{align*}
Therefore, it suffices to take~$\Cr{si-1}$ small enough so that it
satisfies
\[
\Cr{disqQi+1nu} \Cr{si-taylor2}  \Cr{si-atw2}^{2} \Cr{si-key1}^{2} \Cr{si-1} \leq  \a.
\]
This completes the proof of Claim~\ref{claim:inductionstep}.

To finish the proof of Proposition~\ref{prop:newton-nonarhi}, we prove
that the sequence $\{Q_{i}\}$ converges and that its limit has the
desired properties.  In order to talk about the limit, we fix closed
immersions of~$W$ and~$V$ into large affine spaces and identify them
with their images.

By construction,~\eqref{ineq:key2}, and~\eqref{disfQjq}, we have
\begin{align*}
  e^{-\d_{W}(Q_{i},Q_{i+1})}
  &\leq \Cr{si-atw2} \Cr{si-key1} \Cr{si-1}^{1/2} e^{-\d_{V}(\f(Q_{i}),q)/2} \\
  &\leq \Cr{si-atw2} \Cr{si-key1} \Cr{si-1}^{1/2} e^{-\d_{V}(\f(P),q)/2} \a^{i/2}.
\end{align*}
The right-hand side goes to~$0$ as~$i \to \infty$.  Hence~$\{Q_{i}\}$
is a Cauchy sequence (since we are in the non-archimedean setting), so
it has a limit~$Q\in{B}$.  By~\eqref{disfQjq}, we have
\[
\f(Q) = q.
\]

Let~$\iota\colon{W}\subset\AA^{n}$ be our chosen embedding,
so we may define~$\d_W$ by
\[
e^{-\d_{W}(x,x')} = \bigl\| \iota(x) - \iota(x') \bigr\|.
\]
Then, again using the fact that our absolute
value~\text{$|\,\cdot\,|$} is non-archi\-me\-dean, we have
\begin{align*}
e^{-\d_{W}(Q_{i},P)} 
&\leq \max\{ e^{-\d_{W}(Q_{i},Q_{i-1})},  e^{-\d_{W}(Q_{i-1},Q_{i-2})}, \dots, e^{-\d_{W}(Q_{1},P)}   \}\\
&\leq \max\left\{  \Cr{si-key1} \frac{e^{-\d_{V}(\f(Q_{i-1}),q)} }{|J(Q_{i-1})|} ,\dots , \Cr{si-key1} \frac{ e^{-\d_{V}(\f(P),q)}}{|J(P)|}  \right\}\\
&\leq \Cr{si-key1} \frac{1}{|J(P)|} e^{-\d_{V}(\f(P),q)},
\end{align*}
where the second inequality follows from~\eqref{ineq:key1}, and the last inequality follows from~\eqref{invJ} and~\eqref{disfQjq}.
Taking the limit as~$i\to\infty$, we get
\begin{align*}
e^{-\d_{W}(Q ,P)} \leq \Cr{si-key1} \frac{1}{|J(P)|} e^{-\d_{V}(\f(P),q)}
\end{align*}
and we are done.
\end{proof}

\section{Newton's method (Archimedean case)}
\label{section:newtonarch}

In this section we turn to archimedean case and prove the following
result.

\begin{proposition}
\label{prop:newtonmethod-arch}
Let~$\bigl(K,|\,\cdot\,|\bigr)$ be a complete archimedean field, i.e., $K = \RR$ or~$\CC$.
Let~$N>0$ be an integer, let~$b_{2}>b_{1}>0$ be numbers, and define bounded sets
\[
B_{1} = \bigl\{ x \in K^{N} : \|x\| \leq b_{1} \bigr\}\quad\text{and}\quad
B_{2} = \bigl\{ x \in K^{N} : \|x\| \leq b_{2} \bigr\}.
\]
Let~$U \subset K^{N}$ be an open neighborhood of~$B_{2}$, let
\begin{align*}
\f = (\f_{1},\dots, \f_{N}) \colon U \longrightarrow K^{N}
\end{align*}
be an analytic map, and let
\[
D\f = \frac{ \partial (\f_{1},\dots, \f_{N})}{ \partial(x_{1},\dots, x_{N})}
\quad\text{and}\quad
J = \det D\f.
\]
Then there are constants~$\Cl{nmarch-1}, \Cl{nmarch-2}>0$ such that the following holds.
For all~$P\in{B}_{1}$ and~$q\in{K}^{N}$ satisfying
\[
J(P) \neq 0
\quad\text{and}\quad
\bigl\| \f(P) - q \bigr\| \leq \Cr{nmarch-1} \bigl|J(P)\bigr|^{2},
\]
there exists a point~$Q\in{B}_{2}$ satisfying
\[
\f(Q) = q
\quad\text{and}\quad
\|P - Q\| \leq \Cr{nmarch-2} \frac{\bigl\|\f(P) - q\bigr\|}{|J(P)|}.
\]
\end{proposition}

\begin{proof}
In this proof, the labeled constant, always assumed positive, depend
only on~$N$,~$b_{1}$,~$b_{2}$, and~$\f$. And since~$J$ is a continuous
function, there is~$\Cl{nmarch-J}>0$ such that
\begin{equation}
  \label{eqn:Jxbeded}
\bigl|J(x)\bigr| \leq \Cr{nmarch-J}
\quad\text{for all~$x \in B_{2}$.}
\end{equation}

Let~$B'\subset{K}^{N}$ be a sufficiently large bounded subset so that 
\[
\bigcup_{P\in B_1} \Bigl\{ q\in K^N : \bigl\|\f(P) - q\bigr\| \leq \bigl|J(P)\bigr|^{2} \Bigr\}
\subset B'.
\]
We choose a~$\Cl{nmarch-key}>0$ so that for all~$x\in{B}_{2}$
with~$J(x)\neq0$ and all~$q\in{B'}$, we have
\begin{equation}
  \label{ineq:arch-key}
  \bigl\| (D\f)(x)^{-1}(\f(x)-q) \bigr\| \leq \Cr{nmarch-key} \frac{\bigl\| \f(x) - q \bigr\|}{|J(x)|}.
\end{equation}

We fix a small positive number~$0<\e<1$, and we let~$\eta>0$ be a
small positive number that we will specify later. For~$i=1,2,\dots$, we set
\[
c_{i} = \eta^{(2-\e)^{i}},
\]
and we take~$\eta$ sufficiently small to ensure that
\[
b_{1} + \Cr{nmarch-J} \Cr{nmarch-key} \sum_{j\geq 0} c_{j} \leq b_{2}.
\]

Now let~$P\in{B}_{1}$ and~$q\in{B'}$ be points satisfying
\[
J(P) \neq 0
\quad\text{and}\quad
 \bigl\| \f(P) -q \bigr\| \leq \eta \bigl|J(P)\bigr|^{2}.
\]
We start with
\[
Q_{0} = P.
\]
Suppose that we have constructed~$Q_{0},\dots,Q_{i}\in{K}^{N}$
with~$J(Q_{i})\neq0$.  We then define the next point in the sequence
by
\[
Q_{i+1} = Q_{i} - (D\f)(Q_{i})^{-1}\bigl(\f(Q_{i})-q\bigr).
\]
The following claim will be used to show that this sequence converges
to a point having the desired properties.

\begin{claim}
\label{claim:etasmall}
If~$\eta$ is sufficiently small, then the following are true for all~$i=0,1,\dots$.
\begin{parts}
  \Part{(a)}
  $\| Q_{i} \| \leq  b_{1} + \Cr{nmarch-J} \Cr{nmarch-key} \sum_{j = 0}^{i-1} c_{j}$, so in particular,~$Q_{i} \in B_{2}$.
  \Part{(b)}
  $J(Q_{i}) \neq 0$.
  \Part{(c)}
  $\bigl\| \f(Q_{i}) - q  \bigr\| \leq c_{i} \bigl|J(Q_{i})\bigr|^{2}$.
\end{parts}
\end{claim}
\begin{proof}[Proof of Claim $\ref{claim:etasmall}$]
The proof is by induction on~$i$.  If~$i=0$, everything is true by
assumption.  Suppose that~(a),~(b), and~(c) are true for~$i$.

We first calculate 
\begin{align*}
\| Q_{i+1} \| &\leq \| Q_{i}\| + \bigl\|   (D\f)(Q_{i})^{-1}(\f(Q_{i})-q)\bigr\| \\[0mm]
& \leq  \| Q_{i}\| +  \Cr{nmarch-key} \frac{\bigl\| \f(Q_{i}) - q \bigr\|}{\bigl|J(Q_{i})\bigr|}
&& \text{by~\eqref{ineq:arch-key} and~$Q_{i} \in B_{2}$}\\[0mm]
& \leq   \| Q_{i}\| +  \Cr{nmarch-key} c_{i} \bigl|J(Q_{i})\bigr|
&& \text{by (c) for~$i$}\\[0mm]
& \leq \| Q_{i}\| +  \Cr{nmarch-key}  \Cr{nmarch-J} c_{i}
&& \text{since~$Q_{i} \in B_{2}$}\\[0mm]
& \leq b_{1} + \Cr{nmarch-J} \Cr{nmarch-key} \sum_{j = 0}^{i} c_{j}
&& \text{by (a) for~$i$.}
\end{align*}
This proves (a) for~$i+1$.

Since~$\f$ and~$J$ are analytic, there are
constants~$\Cl{nmarch-tayJ},\Cl{nmarch-tayf}>0$ such that for
all~$x,x'\in{B}_{2}$ we have
\begin{align}
\label{arch-tayJ} & \bigl\| J(x') - J(x) \bigr\| \leq \Cr{nmarch-tayJ} \| x'- x\|, \\
\label{arch-tayf} & \bigl\| \f(x') - \f(x) -(D\f)(x)(x'-x)\bigr\| \leq \Cr{nmarch-tayf}\| x'-x \|^{2}.
\end{align}

Since~$Q_{i+1}, Q_{i} \in B_{2}$, by~\eqref{arch-tayJ} we have
\begin{align}
\label{eqn:JxpJxlec89}  
\bigl|J(Q_{i+1}) -J(Q_{i}) \bigr| &\leq \Cr{nmarch-tayJ} |Q_{i+1} - Q_{i}| \notag\\
&= \Cr{nmarch-tayJ} | D\f(Q_{i})^{-1}(\f(Q_{i})-q) | \notag\\
& \leq \Cr{nmarch-tayJ}  \Cr{nmarch-key} \frac{\bigl\| \f(Q_{i}) - q \bigr\|}{\bigl|J(Q_{i})\bigr|}.
\end{align}
Suppose that~$J(Q_{i+1})=0$.
Then we get 
\begin{align*}
\bigl|J(Q_{i})\bigr|^{2} \leq  \Cr{nmarch-tayJ}  \Cr{nmarch-key} \bigl\| \f(Q_{i}) - q \bigr\| \leq \Cr{nmarch-tayJ}  \Cr{nmarch-key} c_{i} \bigl|J(Q_{i})\bigr|^{2}
\leq \Cr{nmarch-tayJ}  \Cr{nmarch-key} \eta \bigl|J(Q_{i})\bigr|^{2}.
\end{align*}
Since~$J(Q_{i}) \neq 0$ by induction hypothesis, this does not happen if~$\eta$ is small enough so that
$\Cr{nmarch-tayJ}  \Cr{nmarch-key} \eta <1$.
This proves (b) for~$i+1$.

We estimate
\begin{align}
   |J(Q_{i+1})| &\geq \bigl|J(Q_{i})\bigr| \left( 1-
  \Cr{nmarch-tayJ} \Cr{nmarch-key} \frac{\bigl\| \f(Q_{i}) - q
    \bigr\|}{\bigl|J(Q_{i})\bigr|^{2}} \right)
  \quad\text{from \eqref{eqn:JxpJxlec89},} \notag\\[0mm]
  \label{arch-Ji+1lower}  & \geq \bigl|J(Q_{i})\bigr| (1- \Cr{nmarch-tayJ}  \Cr{nmarch-key} c_{i})
  \quad\text{from (c) for $i$.}
\end{align}
Further, since~$Q_{i+1},Q_{i}\in{B}_{2}$, we can use~\eqref{arch-tayf} to deduce that
\begin{equation}
  \label{eqn:fQi1fQiDfetc}
  \bigl\|
  \underbrace{ \f(Q_{i+1}) - \f(Q_{i}) - (D\f)(Q_{i})(Q_{i+1} - Q_{i}) }_{\text{By definition of $Q_{i+1}$, this equals $\f(Q_{i+1}) -q$.}}
  \bigr\| \leq \Cr{nmarch-tayf}\| Q_{i+1} - Q_{i} \|^{2}.
\end{equation}
We can estimate the right-hand side of~\eqref{eqn:fQi1fQiDfetc} as follows.

\begin{align*}
\Cr{nmarch-tayf}\| Q_{i+1} - Q_{i} \|^{2}
& = \Cr{nmarch-tayf} \bigl\| (D\f)(Q_{i})^{-1} (\f(Q_{i})-q) \bigr\|^{2}\\[0mm]
&\leq   \Cr{nmarch-tayf}  \Cr{nmarch-key}^{2} \frac{\bigl\| \f(Q_{i}) - q \bigr\|^{2}}{\bigl|J(Q_{i})\bigr|^{2}} \\[0mm]
& \leq \Cr{nmarch-tayf}  \Cr{nmarch-key}^{2} \frac{c_{i}^{2} \bigl|J(Q_{i})\bigr|^{4} }{  \bigl|J(Q_{i})\bigr|^{2} } && \text{by (c) for~$i$,}\\[0mm]
& \leq \Cr{nmarch-tayf}  \Cr{nmarch-key}^{2} \bigl|J(Q_{i})\bigr|^{2} c_{i}^{2}.
\end{align*}
Thus we get
\begin{align}
\label{ineq:arch-dist}  \bigl\| \f(Q_{i+1}) -q \bigr\| & \leq  \Cr{nmarch-tayf}  \Cr{nmarch-key}^{2} \frac{\bigl\| \f(Q_{i}) - q \bigr\|^{2}}{\bigl|J(Q_{i})\bigr|^{2}}  \\[0mm]
\notag & \leq \Cr{nmarch-tayf}  \Cr{nmarch-key}^{2} \bigl|J(Q_{i})\bigr|^{2} c_{i}^{2}.
\end{align}
By~\eqref{arch-Ji+1lower}, we have
\begin{align*}
 \bigl\| \f(Q_{i+1}) -q \bigr\| & \leq \Cr{nmarch-tayf}  \Cr{nmarch-key}^{2}  \frac{\bigl|J(Q_{i+1})\bigr|^{2}}{(1- \Cr{nmarch-tayJ}  \Cr{nmarch-key} c_{i})^{2}}c_{i}^{2}\\[0mm]
 & = \frac{ \Cr{nmarch-tayf}  \Cr{nmarch-key}^{2}}{(1- \Cr{nmarch-tayJ}  \Cr{nmarch-key} c_{i})^{2}} \eta^{\e(2-\e)^{i}} \bigl|J(Q_{i+1})\bigr|^{2} \eta^{(2-\e)^{i+1}} \\[0mm]
 & = \frac{ \Cr{nmarch-tayf}  \Cr{nmarch-key}^{2}}{(1- \Cr{nmarch-tayJ}  \Cr{nmarch-key} c_{i})^{2}} \eta^{\e(2-\e)^{i}} \bigl|J(Q_{i+1})\bigr|^{2} c_{i+1}  \\[0mm]
 & \leq \frac{ \Cr{nmarch-tayf}  \Cr{nmarch-key}^{2}}{(1- \Cr{nmarch-tayJ}  \Cr{nmarch-key} \eta)^{2}} \eta^{\e} \bigl|J(Q_{i+1})\bigr|^{2} c_{i+1} 
\end{align*}
Thus if we choose~$\eta$ small enough so that it satisfies
\begin{align*}
\frac{ \Cr{nmarch-tayf}  \Cr{nmarch-key}^{2}}{(1- \Cr{nmarch-tayJ}  \Cr{nmarch-key} \eta)^{2}} \eta^{\e} \leq 1,
\end{align*}
then we get
\begin{align*}
 \bigl\| \f(Q_{i+1}) -q \bigr\|  \leq \bigl|J(Q_{i+1})\bigr|^{2} c_{i+1},
\end{align*}
and we are done with the proof of Claim~\ref{claim:etasmall}.
\end{proof}

We can use Claim~\ref{claim:etasmall} to prove that the
sequence~$Q_{i}$ is a Cauchy sequence via the following calculation:
\begin{align*}
  \| Q_n - Q_m \|
  &\le \sum_{i=m}^{n-1} \|Q_{i+1} - Q_i \| \\
  &\le \sum_{i=m}^{n-1} \Bigl( \|Q_{i+1}-q\| + \|Q_i-q\| \Bigr) \\
  &\le \sum_{i=m}^{n-1} \Bigl( c_{i+1}\bigl|J(Q_{i+1})\bigr|^2 + c_{i}\bigl|J(Q_{i})\bigr|^2 \Bigr)
  \quad\text{from Claim~\ref{claim:etasmall}(c),} \\
  &\le \sum_{i=m}^{n-1} \Cr{nmarch-J}^2( c_{i+1}+ c_{i} ) 
  \quad\text{from \eqref{eqn:Jxbeded}, since $Q_i,Q_{i+1}\in B_2$,}  \\
  &\xrightarrow[n,m\to\infty]{} 0
  \quad\text{since $c_i=\eta^{(2-\e)^i}$ and $0<\eta<1$.}
\end{align*}
Hence the limit
\[
Q := \lim_{i \to \infty}Q_{i} \in B_{2}
\]
exists. Further, using the continuity of~$D\f$ and the definition
of~$Q_{i+1}$, we see that
\[
\lim_{i\to\infty} \f(Q_i)-q
= \lim_{i\to\infty} (D\f)(Q_i) (Q_{i+1}-Q_i)
= (D\f)(Q) (Q-Q) = 0,
\]
which proves that
\[
\f(Q) = \lim_{i\to \infty} \f(Q_{i}) = q.
\]

It remains to show that~$Q$ is close to~$P$.
We have
\[
\|Q - P\| = \| Q - Q_{0}\| \leq \sum_{i\geq 0} \Cr{nmarch-key} \frac{\bigl\| \f(Q_{i}) - q \bigr\|}{\bigl|J(Q_{i})\bigr|}.
\]
Using~\eqref{arch-Ji+1lower} and~\eqref{ineq:arch-dist} gives
\begin{align*}
  \frac{\bigl\| \f(Q_{i+1}) - q \bigr\|}{ \bigl|J(Q_{i+1})\bigr| }
  &\leq \Cr{nmarch-tayf}  \Cr{nmarch-key}^{2} \frac{\bigl\| \f(Q_{i}) - q \bigr\|^{2}}{\bigl|J(Q_{i})\bigr|^{2}}
  \frac{1}{ \bigl|J(Q_{i})\bigr| (1- \Cr{nmarch-tayJ}  \Cr{nmarch-key} c_{i})}\\[0mm]
  & \leq \frac{\Cr{nmarch-tayf}  \Cr{nmarch-key}^{2} }{1- \Cr{nmarch-tayJ}  \Cr{nmarch-key} c_{i}} 
  \cdot c_{i} \cdot\frac{\bigl\| \f(Q_{i}) - q \bigr\|}{ \bigl|J(Q_{i})\bigr| } \\[0mm]
  & \leq \frac{\Cr{nmarch-tayf}  \Cr{nmarch-key}^{2} }{1- \Cr{nmarch-tayJ}  \Cr{nmarch-key} \eta}
  \cdot\eta\cdot \frac{\bigl\| \f(Q_{i}) - q \bigr\|}{ \bigl|J(Q_{i})\bigr| }. 
\end{align*}
We take~$\eta$ sufficiently small so that we have
\[
\tau := \frac{\Cr{nmarch-tayf}  \Cr{nmarch-key}^{2} }{1- \Cr{nmarch-tayJ}  \Cr{nmarch-key} \eta} \eta < 1.
\]
Then we get
\[
\frac{\bigl\| \f(Q_{i}) - q \bigr\|}{ \bigl|J(Q_{i})\bigr| } \leq \tau^{i} \frac{\bigl\| \f(P) - q \bigr\|}{ |J(P)| },
\]
and hence
\[
\|Q - P\| \leq \Cr{nmarch-key} \frac{1}{1 - \tau} \frac{\bigl\| \f(P) - q \bigr\|}{ |J(P)| }, 
\]
which completes the proof of Proposition~\ref{prop:newtonmethod-arch}.
\end{proof}






\end{document}